\documentclass{article}
\usepackage{amsmath}
\usepackage{amssymb}
\usepackage{amsthm}
\usepackage{graphicx}
\usepackage{cite}
\usepackage[arrow,curve,matrix,arc,2cell]{xy}
\UseAllTwocells
\usepackage[utf8]{inputenc}
\usepackage[unicode]{hyperref}
\DeclareFontFamily{U}{rsfs}{} \DeclareFontShape{U}{rsfs}{n}{it}{<->
rsfs10}{} \DeclareSymbolFont{mscr}{U}{rsfs}{n}{it}
\DeclareSymbolFontAlphabet{\scr}{mscr}
\def\mathscr{\scr}
\begin{document}
\def\e#1\e{\begin{equation}#1\end{equation}}
\def\ea#1\ea{\begin{align}#1\end{align}}
\def\eq#1{{\rm(\ref{#1})}}
\theoremstyle{plain}
\newtheorem{thm}{Theorem}[section]
\newtheorem{prop}[thm]{Proposition}
\newtheorem{lem}[thm]{Lemma}
\newtheorem{cor}[thm]{Corollary}
\newtheorem{quest}[thm]{Question}
\theoremstyle{definition}
\newtheorem{dfn}[thm]{Definition}
\newtheorem{ex}[thm]{Example}
\newtheorem{rem}[thm]{Remark}
\numberwithin{equation}{section}
\def\dim{\mathop{\rm dim}\nolimits}
\def\supp{\mathop{\rm supp}\nolimits}
\def\cosupp{\mathop{\rm cosupp}\nolimits}
\def\id{\mathop{\rm id}\nolimits}
\def\Hess{\mathop{\rm Hess}\nolimits}
\def\Crit{\mathop{\rm Crit}}
\def\an{{\rm an}}
\def\Aut{\mathop{\rm Aut}}
\def\coh{\mathop{\rm coh}\nolimits}
\def\Hom{\mathop{\rm Hom}\nolimits}
\def\Perv{\mathop{\rm Perv}\nolimits}
\def\Sch{\mathop{\rm Sch}\nolimits}
\def\Var{\mathop{\rm Var}\nolimits}
\def\red{{\rm red}}
\def\DM{\mathop{\rm DM}\nolimits}
\def\Spec{\mathop{\rm Spec}\nolimits}
\def\rank{\mathop{\rm rank}\nolimits}
\def\bs{\boldsymbol}
\def\ge{\geqslant}
\def\le{\leqslant\nobreak}
\def\bA{{\mathbin{\mathbb A}}}
\def\bL{{\mathbin{\mathbb L}}}
\def\cE{{\mathbin{\cal E}}}
\def\cL{{\mathbin{\cal L}}}
\def\cM{{\mathbin{\cal M}}}
\def\oM{{\mathbin{\smash{\,\,\overline{\!\!\mathcal M\!}\,}}}}
\def\O{{\mathbin{\cal O}}}
\def\PV{{\mathbin{\cal{PV}}}}
\def\bG{{\mathbin{\mathbb G}}}
\def\fE{{\mathbin{\mathfrak E}}}
\def\fG{{\mathbin{\mathfrak G}}}
\def\fU{\mathbin{\mathfrak U}}
\def\C{{\mathbin{\mathbb C}}}
\def\K{{\mathbin{\mathbb K}}}
\def\N{{\mathbin{\mathbb N}}}
\def\Q{{\mathbin{\mathbb Q}}}
\def\Z{{\mathbin{\mathbb Z}}}
\def\al{\alpha}
\def\be{\beta}
\def\ga{\gamma}
\def\de{\delta}
\def\io{\iota}
\def\ep{\epsilon}
\def\la{\lambda}
\def\ka{\kappa}
\def\th{\theta}
\def\ze{\zeta}
\def\up{\upsilon}
\def\vp{\varphi}
\def\si{\sigma}
\def\om{\omega}
\def\De{\Delta}
\def\La{\Lambda}
\def\Th{\Theta}
\def\Om{\Omega}
\def\Ga{\Gamma}
\def\Si{\Sigma}
\def\Up{\Upsilon}
\def\ts{\textstyle}
\def\st{\scriptstyle}
\def\sst{\scriptscriptstyle}
\def\sm{\setminus}
\def\bu{\bullet}
\def\op{\oplus}
\def\ot{\otimes}
\def\boxt{\boxtimes}
\def\ov{\overline}
\def\ul{\underline}
\def\bigop{\bigoplus}
\def\iy{\infty}
\def\es{\emptyset}
\def\ra{\rightarrow}
\def\ab{\allowbreak}
\def\longra{\longrightarrow}
\def\hookra{\hookrightarrow}
\def\bs{\boldsymbol}
\def\t{\times}
\def\ci{\circ}
\def\ti{\tilde}
\def\d{{\rm d}}
\def\od{\odot}
\def\bd{\boxdot}
\def\md#1{\vert #1 \vert}
\def\cS{{\mathbin{\cal S}}}
\def\cSz{{\mathbin{\cal S}\kern -0.1em}^{\kern .1em 0}}
\title{On motivic vanishing cycles of critical loci}
\author{Vittoria Bussi, Dominic Joyce, and Sven Meinhardt}
\date{}
\maketitle

\begin{abstract}
Let $U$ be a smooth scheme over an algebraically closed field $\K$
of characteristic zero and $f:U\ra\bA^1$ a regular function, and
write $X=\Crit(f)$, as a closed $\K$-subscheme of $U$. The {\it
motivic vanishing cycle\/} $MF_{U,f}^{\rm mot,\phi}$ is an element
of the $\hat\mu$-equivariant motivic Grothendieck ring
$\smash{\cM^{\hat\mu}_X}$, defined by Denef and Loeser
\cite{DeLo3,DeLo2} and Looijenga \cite{Looi}, and used in Kontsevich
and Soibelman's theory of motivic Donaldson--Thomas invariants
\cite{KoSo1}.

We prove three main results:
\smallskip

\noindent{\bf(a)} $MF_{U,f}^{\rm mot,\phi}\!$ depends only on the
third-order thickenings $U^{(3)}\!,f^{(3)}$~of~$U,f$.

\smallskip

\noindent{\bf(b)} If $V$ is another smooth $\K$-scheme,
$g:V\ra\bA^1$ is regular, $Y=\Crit(g)$, and $\Phi:U\ra V$ is an
embedding with $f=g\ci\Phi$ and $\Phi\vert_X:X\ra Y$ an isomorphism,
then $\Phi\vert_X^*\bigl(MF^{\rm mot, \phi}_{V,g}\bigr)=MF^{\rm mot,
\phi}_{U,f}\od\Up(P_\Phi)$ in a certain quotient ring
$\oM^{\hat\mu}_X$ of $\cM^{\hat\mu}_X$, where $P_\Phi\ra X$ is a
principal $\Z_2$-bundle associated to $\Phi$ and $\Up:\{$principal
$\Z_2$-bundles on $X\}\ra\oM^{\hat\mu}_X$ a natural morphism.
\smallskip

\noindent{\bf(c)} If $(X,s)$ is an {\it oriented algebraic
d-critical locus\/} in the sense of Joyce \cite{Joyc2}, there is a
natural motive $MF_{X,s}\in\oM^{\hat\mu}_X$, such that if $(X,s)$ is
locally modelled on $\Crit(f:U\ra\bA^1)$, then $MF_{X,s}$ is locally
modelled on~$MF_{U,f}^{\rm mot,\phi}$.

\smallskip

Using results of Pantev, To\"en, Vezzosi and Vaqui\'e \cite{PTVV},
these imply the existence of natural motives on moduli schemes of
coherent sheaves on a Calabi--Yau 3-fold equipped with `orientation
data', as required in Kontsevich and Soibelman's motivic
Donaldson--Thomas theory \cite{KoSo1}, and on intersections $L\cap
M$ of oriented Lagrangians $L,M$ in an algebraic symplectic manifold
$(S,\om)$.
\end{abstract}

\setcounter{tocdepth}{2}
\tableofcontents

\section{Introduction}
\label{mo1}

Brav, Bussi, Dupont, Joyce and Szendr\H oi \cite{BBDJS} proved some
results on perverse sheaves, $\scr D$-modules and mixed Hodge
modules of vanishing cycles on critical loci, and gave some
applications to categorification of Donaldson--Thomas invariants of
Calabi--Yau 3-folds, and to defining `Fukaya categories' of complex
or algebraic symplectic manifolds using perverse sheaves. This paper
is a sequel to \cite{BBDJS}, in which we prove analogous results for
motives and motivic vanishing cycles, with applications to motivic
Donaldson--Thomas invariants.

Let $\K$ be an algebraically closed field of characteristic zero,
$U$ a smooth $\K$-scheme, $f:U\ra\bA^1$ a regular function, and
$U_0=f^{-1}(0)$, $X=\Crit(f)$ as closed $\K$-subschemes of $U$.
Following Denef and Loeser \cite{DeLo2,DeLo3} and Looijenga
\cite{Looi}, in \S\ref{mo2} we will define the {\it motivic nearby
cycle\/} $MF_{U,f}^{\rm mot}$ in the monodromic Grothendieck group
$\smash{K_0^{\hat\mu}(U_0)}$ of $\hat\mu$-equivariant motives on
$U_0$, and the {\it motivic vanishing cycle\/} $MF_{U,f}^{\rm
mot,\phi}$ in the ring $\cM_X^{\hat\mu}=K_0^{\hat\mu}(X)[\bL^{-1}]$
with Tate motive $\bL=[\bA^1]$ inverted.

Here $\smash{MF_{U,f}^{\rm mot}}$ is the motivic analogue of the
constructible complex of nearby cycles $\psi_f(\Q_U)\in\Perv(U_0)$
in \cite{BBDJS}, and $MF_{U,f}^{\rm mot,\phi}$ the motivic analogue
of the perverse sheaf of vanishing cycles $\PV_{U,f}^\bu=
\phi_f(\Q_U[\dim U-1])\in\Perv(X)$ in \cite{BBDJS} (at least when
$X\subseteq U_0$). The fibre $MF_{U,f}^{\rm mot}(x)$ of
$MF_{U,f}^{\rm mot}$ at $x\in U_0$ is the {\it motivic Milnor
fibre\/} of $f$ at $x$ from \cite{DeLo2,DeLo3,Looi}, the algebraic
analogue of the Milnor fibre $MF_f(x)$ at $x$ of a holomorphic
function $f:U\ra\C$ on a complex manifold~$U$.

We will prove three main results, Theorems \ref{mo3thm1},
\ref{mo4thm2} and \ref{mo5thm5}. The first, Theorem \ref{mo3thm1},
says that $MF_{U,f}^{\rm mot,\phi}\in\cM_X^{\hat\mu}$ depends only
on the third-order thickenings $U^{(3)},f^{(3)}$ of $U,f$ at $X$,
where $\O_{\smash{U^{(3)}}}=\O_U/I_X^3$, for $I_X\subseteq\O_U$ the
ideal of functions $U\ra\bA^1$ vanishing on $X$, and
$f^{(3)}=f\vert_{\smash{U^{(3)}}}$. We also show by example that
$U^{(2)},f^{(2)}$ do not determine $MF_{U,f}^{\rm mot,\phi}$.

Our second and third main results involve {\it principal\/
$\Z_2$-bundles\/} $P\ra X$ over a $\K$-scheme $X$. In \S\ref{mo25}
we define a natural motive $\Up(P)\in\cM_X^{\hat\mu}$ for each
principal $\Z_2$-bundle $P\ra X$. As in Denef and Loeser
\cite{DeLo3} and Looijenga \cite{Looi}, there is a (non-obvious)
commutative, associative multiplication $\od$ on $\cM_X^{\hat\mu}$
which appears in the motivic Thom--Sebastiani Theorem
\cite{DeLo2,DeLo3,Looi}.

In \S\ref{mo4}--\S\ref{mo5} we need the $\Up(P)$ to satisfy
$\Up(P\ot_{\Z_2}Q)=\Up(P)\od\Up(Q)$ for all principal $\Z_2$-bundles
$P,Q\ra X$, but we cannot prove this in $\cM_X^{\hat\mu}$. Our
solution in \S\ref{mo25} is to define a new ring of motives
$\oM^{\hat\mu}_Y$ for each $\K$-scheme $Y$ to be the quotient of
$\bigl(\cM_Y^{\hat\mu},\od\bigr)$ by the ideal generated by pushforwards
$\phi_*\bigl(\Up(P\ot_{\Z_2}Q)-\Up(P)\od\Up(Q)\bigr)$ for all $\K$-scheme morphisms $\phi:X\ra Y$ and principal $\Z_2$-bundles $P,Q\ra X$, and then
$\Up(P\ot_{\Z_2}Q)=\Up(P)\od\Up(Q)$ holds in~$\oM^{\hat\mu}_X$.

Essentially the same issue occurs in Kontsevich and Soibelman
\cite{KoSo1}, which inspired this part of our paper. In defining the
motivic rings $\oM^\mu(X)$ in which their motivic Donaldson--Thomas
invariants take values, in \cite[\S 4.5]{KoSo1} they impose a
complicated relation, which as in \cite[\S 5.1]{KoSo1} implies that
the motivic vanishing cycle $MF_{E,q}^{\rm mot,\phi}$ of a
nondegenerate quadratic form $q$ on a vector bundle $E\ra U$ depends
only on the triple $\bigl(\rank E,\La^{\rm top}E,\det q\bigr)$. As
in \S\ref{mo25}, this implies our relation $\Up(P\ot_{\Z_2}Q)=\Up(P)
\od\Up(Q)$. So Kontsevich and Soibelman's ring $\oM^\mu(X)$ is a
quotient of our ring~$\oM^{\hat\mu}_X$.

Our second main result, Theorem \ref{mo4thm2}, says that if $U,V$
are smooth $\K$-schemes, $f:U\ra\bA^1$, $g:V\ra\bA^1$ are regular,
$X=\Crit(f)$, $Y=\Crit(g)$, and $\Phi:U\ra V$ is an embedding with
$f=g\ci\Phi$ and $\Phi\vert_X:X\ra Y$ an isomorphism, then
$\Phi\vert_X^*\bigl(MF^{\rm mot,\phi}_{V,g}\bigr)=MF^{\rm mot,
\phi}_{U,f}\od\Up(P_\Phi)$ in $\oM^{\hat\mu}_X$, for $P_\Phi\ra X$ a
principal $\Z_2$-bundle parametrizing orientations of the
nondegenerate quadratic form $\Hess g$ on $N_{\sst UV}\vert_X$, with
$N_{\sst UV}\ra U$ the normal bundle of $\Phi(U)$ in $V$. The
analogous result \cite[Th.~5.4]{BBDJS} for perverse sheaves of
vanishing cycles $\PV_{U,f}^\bu$ says that~$\Phi\vert_X^*
(\PV_{V,g}^\bu)\cong\PV_{U,f}^\bu\ot_{\Z_2}P_\Phi$.

For $U,V,f,g,\Phi$ as above, \cite[Prop.~2.23]{Joyc2} shows that
\'etale locally on $V$ we have equivalences $V\sim U\t\bA^n$
identifying $g\sim f\boxplus z_1^2+\cdots+z_n^2$ and $\Phi\sim
\id_U\t 0$. So if we could work \'etale locally, we would have
\begin{align*}
\Phi\vert_X^*&\bigl(MF^{\rm mot,\phi}_{V,g}\bigr)\sim
(\id_X\t 0)^*\bigl(MF^{\rm mot,\phi}_{U\t\bA^n,f\boxplus
z_1^2+\cdots+z_n^2}\bigr)\\
&\quad=MF^{\rm mot,\phi}_{U,f}\bd MF^{\rm mot,\phi}_{\bA^n,z_1^2+\cdots+
z_n^2}=MF^{\rm mot,\phi}_{U,f}\bd 1_{\{0\}}=MF^{\rm mot,\phi}_{U,f},
\end{align*}
using the motivic Thom--Sebastiani theorem in the second step.
However, for motives we must work Zariski locally, so we need a more
complicated proof involving the (\'etale locally trivial) correction
factor $\Up(P_\Phi)$. In singularity theory, passing from $f$ to
$f\boxplus z_1^2+\cdots+z_n^2$ is known as {\it stabilization}, so
Theorem \ref{mo4thm2} studies the behaviour of motivic vanishing
cycles under stabilization.

Our third main result, Theorem \ref{mo5thm5}, concerns a new class
of geometric objects called {\it d-critical loci}, introduced in
Joyce \cite{Joyc2}, and explained in \S\ref{mo51}. An (algebraic)
d-critical locus $(X,s)$ over $\K$ is a $\K$-scheme $X$ with a
section $s$ of a certain natural sheaf $\cSz_X$ on $X$. A d-critical
locus $(X,s)$ may be written Zariski locally as a critical locus
$\Crit(f:U\ra\bA^1)$ of a regular function $f$ on a smooth
$\K$-scheme $U$, and $s$ records some information about $U,f$ (in
the notation of our first main result, $s$ remembers $f^{(2)}$).
There is also a complex analytic version.

Algebraic d-critical loci are classical truncations of the {\it
derived critical loci\/} (more precisely, $-1$-{\it shifted
symplectic derived schemes\/}) introduced in derived algebraic
geometry by Pantev, To\"en, Vaqui\'e and Vezzosi \cite{PTVV}.
Theorem \ref{mo5thm5} roughly says that if $(X,s)$ is an algebraic
d-critical locus over $\K$ with an `orientation', then we may define
a natural motive $MF_{X,s}$ in $\oM^{\hat\mu}_X$, such that if
$(X,s)$ is locally modelled on $\Crit(f:U\ra\bA^1)$ then $MF_{X,s}$
is locally modelled on $MF_{U,f}^{\rm mot,\phi}\od\Up(P)$, where
$P\ra X$ is a principal $\Z_2$-bundle relating the `orientations' on
$(X,s)$ and $\Crit(f)$. The proof uses Theorem~\ref{mo4thm2}.

Bussi, Brav and Joyce \cite{BBJ} prove Darboux-type theorems for the
$k$-shifted symplectic derived schemes of Pantev et al.\
\cite{PTVV}, and use them to construct a truncation functor from
$-1$-shifted symplectic derived schemes to algebraic d-critical
loci. Combining this with results of \cite{BBDJS,Joyc2,PTVV} and
this paper gives new results on categorifying Donaldson--Thomas
invariants of Calabi--Yau 3-folds, and on motivic Donaldson--Thomas
invariants. In particular, as we explain in \S\ref{mo52}, Kontsevich
and Soibelman \cite{KoSo1} wish to associate a motivic Milnor fibre
to each point of the moduli $\K$-schemes $\cM_{\rm st}^\al(\tau)$ of
$\tau$-stable coherent sheaves on a Calabi--Yau $3$-fold over $\K$.
The issue of how these vary in families over the base $\cM_{\rm
st}^\al(\tau)$ is not really addressed in \cite{KoSo1}. Our paper
answers this question.

In the rest of the paper, \S\ref{mo2} introduces motivic Milnor
fibres and motivic vanishing cycles, and \S\ref{mo3}--\S\ref{mo5}
state and prove Theorems \ref{mo3thm1}, \ref{mo4thm2} and~\ref{mo5thm5}.

Ben-Bassat, Brav, Bussi and Joyce \cite{BBBJ} will extend the
results of \cite{BBJ,BBDJS} and this paper from (derived) schemes to
(derived) Artin stacks.
\medskip

\noindent{\bf Conventions.} Throughout we work over a base field
$\K$ which is algebraically closed and of characteristic zero, for
instance $\K=\C$. All $\K$-schemes are assumed to be of finite type,
unless we explicitly say otherwise. We discuss extending our results to $\K$-schemes locally of finite type in Remark~\ref{mo5rem}.
\medskip

\noindent{\bf Acknowledgements.} We would like to thank Ben Davison,
Johannes Nicaise, Le Quy Thuong and Davesh Maulik for useful conversations, and the referee for helpful comments. This research was supported by EPSRC Programme Grant EP/I033343/1 on `Motivic invariants and categorification'.

\section{Background material}
\label{mo2}

We begin by discussing rings of motives, motivic Milnor fibres, and
motivic vanishing cycles. Sections \ref{mo21}--\ref{mo23} broadly
follow Denef and Loeser \cite{DeLo3,DeLo4,DeLo1,DeLo2} and Looijenga
\cite{Looi}. Sections \ref{mo24}--\ref{mo25} contain some new
material, much of which is based on ideas in Kontsevich and
Soibelman~\cite{KoSo1}.

\subsection{\texorpdfstring{Rings of motives on a $\K$-scheme $X$}{Rings of motives on a 𝕂-scheme X}}
\label{mo21}

We define ($\hat\mu$-equivariant) Grothendieck groups of schemes.

\begin{dfn} Let $X$ be a $\K$-scheme (always assumed of finite
type). By an $X$-scheme we mean a $\K$-scheme $S$ together with a
morphism $\Pi_S^X:S\ra X$. The $X$-schemes form a category $\Sch_X$,
with morphisms $\al:S\ra T$ satisfying $\Pi_S^X=\Pi_T^X\ci\al$.

Write $K_0(\Sch_X)$ for the {\it Grothendieck group\/} of $\Sch_X$.
It is the abelian group generated by symbols $[S]$, for $S$ an
$X$-scheme, with relations $[S]=[T]$ if $S\cong T$ in $\Sch_X$, and
$[S]=[T]+[S\sm T]$ if $T\subseteq S$ is a closed $X$-subscheme.
There is a natural commutative ring structure on $K_0(\Sch_X)$, with
$[S]\cdot[T]=[S\t_X T]$.

Write $\bA^1_X$ for the $X$-scheme $\pi_X:\bA^1\t X\ra X$, and
define $\bL=[\bA^1_X]$ in $K_0(\Sch_X)$. We denote by
$\cM_X=K_0(\Sch_X)[\bL^{-1}]$ the ring obtained from $K_0(\Sch_X)$
by inverting $\bL$. When $X=\Spec\K$ we write $K_0(\Sch_\K),\cM_\K$
instead of $K_0(\Sch_X),\cM_X$. If $X,Y$ are $\K$-schemes there are
natural {\it external products\/}
\e
\begin{gathered}
\boxt:K_0(\Sch_X)\t K_0(\Sch_Y)\ra K_0(\Sch_{X\t Y}),\;\>
\boxt:\cM_X\t\cM_Y\ra\cM_{X\t Y}\\
\text{with}\;\> \bigl[\Pi_S^X:S\ra X\bigr]\boxt\bigl[\Pi_T^Y:T\ra
Y\bigr]=\bigl[\Pi_S^X\t\Pi_T^Y:S\t T\ra X\t Y\bigr].
\end{gathered}
\label{mo2eq1}
\e

If $f:X\ra Y$ is a morphism of $\K$-schemes, we define {\it
pushforwards\/} $f_*:K_0(\Sch_X)\ra K_0(\Sch_Y)$,
$f_*:\cM_X\ra\cM_Y$ and {\it pullbacks\/} $f^*:K_0(\Sch_Y)\ra
K_0(\Sch_X)$, $f^*:\cM_Y\ra\cM_X$ by $f_*\bigl(\bigl[ \Pi_S^X:S\ra
X\bigr]\bigr)=\bigl[f\ci \Pi_S^X:S\ra Y\bigr]$ and
$f^*\bigl(\bigl[\Pi_S^Y:S\ra Y\bigr]\bigr)=\bigl[\pi_X:S\t_{\Pi_S^Y,
Y,f}X\ra X\bigr]$. They have the usual functorial properties.
\label{mo2def1}
\end{dfn}

\begin{dfn} For $n=1,2,\ldots,$ write $\mu_n$ for the group
of all $n^{\rm th}$ roots of unity in $\K$, which is assumed
algebraically closed of characteristic zero as in \S\ref{mo1}, so
that $\mu_n\cong\Z_n$. Then $\mu_n$ is the $\K$-scheme
$\Spec(\K[x]/(x^n-1))$. The $\mu_n$ form a projective system, with
respect to the maps $\mu_{nd}\ra\mu_n$ mapping $x\mapsto x^d$ for
all $d,n=1,2,\ldots.$ Define the group $\hat\mu$ to be the
projective limit of the $\mu_n$. Note that $\hat\mu$ is not a
$\K$-scheme, but is a pro-scheme.

Let $S$ be an $X$-scheme. A {\it good\/ $\mu_n$-action on\/} $S$ is
a group action $\si_n:\mu_n\t S\ra S$ which is a morphism of
$X$-schemes, such that each orbit is contained in an affine
subscheme of $S$. This last condition is automatically satisfied
when $S$ is quasi-projective. A {\it good\/ $\hat\mu$-action on\/}
$S$ is a group action $\hat\si:\hat\mu\t S\ra S$ which factors
through a good $\mu_n$-action, for some $n$. We will write
$\hat\io:\hat\mu\t S\ra S$ for the trivial $\hat\mu$-action on $S$,
for any $S$, which is automatically good.

The {\it monodromic Grothendieck group} $K^{\hat\mu}_0(\Sch_X)$ is
the abelian group generated by symbols $[S,\hat\si]$, for $S$ an
$X$-scheme and $\hat\si:\hat\mu\t S\ra S$ a good $\hat\mu$-action,
with the relations:
\begin{itemize}
\setlength{\itemsep}{0pt}
\setlength{\parsep}{0pt}
\item[(i)] $[S,\hat\si]=[T,\hat\tau]$ if $S,T$ are isomorphic as
$X$-schemes with $\hat\mu$-actions;
\item[(ii)] $[S,\hat \si]=[T,\hat\si\vert_T]+[S\sm
T,\hat\si\vert_{S\sm T}]$ if $T\subseteq S$ is a closed,
$\hat\mu$-invariant $X$-subscheme of $S$; and
\item[(iii)] $[S\t\bA^n,\hat\si\t\hat\tau_1]=
[S\t \bA^n,\hat\si\t\hat\tau_2]$ for any linear $\hat\mu$-actions
$\hat\tau_1,\hat\tau_2$ on $\bA^n$.
\end{itemize}
There is a natural commutative ring structure on $K^{\hat\mu}_0
(\Sch_X)$ with multiplication `$\,\cdot\,$' defined by
$[S,\hat\si]\cdot[T,\hat\tau]=[S\t_XT,\hat\si \t\hat\tau]$. Write
$\bL=[\bA^1_X,\hat\io]$ in $K_0^{\hat\mu}(\Sch_X)$. Define
$\cM^{\hat\mu}_X= K_0^{\hat\mu}(\Sch_X)[\bL^{-1}]$ to be the ring
obtained from $K_0^{\hat\mu}(\Sch_X)$ (with multiplication
`$\,\cdot\,$') by inverting $\bL$. When $X=\Spec\K$ we write
$K^{\hat\mu}_0(\Sch_\K),\cM^{\hat\mu}_\K$ instead of
$K^{\hat\mu}_0(\Sch_X),\cM^{\hat\mu}_X$.

If $X,Y$ are $\K$-schemes there are natural {\it external
products\/}
\begin{gather}
\boxt:K^{\hat\mu}_0(\Sch_X)\t K^{\hat\mu}_0(\Sch_Y)\ra
K^{\hat\mu}_0(\Sch_{X\t Y}),\;\>
\boxt:\cM^{\hat\mu}_X\t\cM^{\hat\mu}_Y \ra \cM^{\hat\mu}_{X\t Y}
\nonumber\\
\text{with}\;\> [S\ra X,\hat\si]\boxt[T\ra Y,\hat\tau]=[S\t T\ra X\t
Y,\hat\si\t\hat\tau].
\label{mo2eq2}
\end{gather}
Pushforwards and pullbacks work for the rings $\cM^{\hat\mu}_X$ in
the obvious way.

There are also natural morphisms of commutative rings
\begin{align*}
i_X:K_0(\Sch_X)&\longra K_0^{\hat\mu}(\Sch_X),&
i_X:\cM_X&\longra\cM^{\hat\mu}_X, & i_X:[S]&\longmapsto[S,\hat\io],\\
\Pi_X:K_0^{\hat\mu}(\Sch_X)&\longra K_0(\Sch_X),&
\Pi_X:\cM^{\hat\mu}_X&\longra\cM_X, & \Pi_X:[S,\hat\si]&\longmapsto[S].
\end{align*}

\label{mo2def2}
\end{dfn}

Following Looijenga \cite[\S 7]{Looi} and Denef and Loeser \cite[\S
5]{DeLo3}, we introduce a second commutative, associative
multiplication `$\od$' on $K_0^{\hat\mu}(\Sch_X),\cM^{\hat\mu}_X$
(written `$*$' in \cite{Looi,DeLo3}).

\begin{dfn} Let $X$ be a $\K$-scheme, and
$[S,\hat\si],[T,\hat\tau]$ generators of $K^{\hat\mu}_0(\Sch_X)$ or
$\cM^{\hat\mu}_X$. Then there exists $n\ge 1$ such that the
$\hat\mu$-actions $\hat\si,\hat\tau$ on $S,T$ factor through
$\mu_n$-actions $\si_n,\tau_n$. Define $J_n$ to be the Fermat curve
\begin{equation*}
J_n=\bigl\{(t,u)\in(\bA^1\sm\{0\})^2:t^n+u^n=1\bigr\}.
\end{equation*}
Let $\mu_n\t\mu_n$ act on $J_n\t (S\t_X T)$ by
\begin{equation*}
(\al,\al')\cdot\bigl((t,u),(v,w)\bigr)=\bigl((\al\cdot t,\al'\cdot
u),(\si_n(\al)(v),\tau_n(\al')(w))\bigr).
\end{equation*}
Write $J_n(S,T)=(J_n\t (S\t_X T))/(\mu_n\t\mu_n)$ for the quotient
$\K$-scheme, and define a $\mu_n$-action $\up_n$ on $J_n(S,T)$ by
\begin{equation*}
\up_n(\al)\bigl((t,u),v,w\bigr)(\mu_n\t\mu_n)= \bigl((\al\cdot
t,\al\cdot u),v,w\bigr)(\mu_n\t\mu_n).
\end{equation*}
Let $\hat\up$ be the induced good $\hat\mu$-action on $J_n(S,T)$,
and set
\e
[S,\hat\si]\od[T,\hat\tau]=(\bL-1)\cdot\bigl[(S\t_X
T)/\mu_n,\hat\io\bigr]- \bigl[J_n(S,T),\hat\up\bigr]
\label{mo2eq3}
\e
in $K_0^{\hat\mu}(\Sch_X)$ or $\cM^{\hat\mu}_X$. This turns out to
be independent of $n$, and defines commutative, associative products
$\od$ on $K_0^{\hat\mu}(\Sch_X),\cM^{\hat\mu}_X$.

Now let $X,Y$ be $\K$-schemes. As for \eq{mo2eq1}--\eq{mo2eq2}, we
define products
\begin{equation*}
\bd:K^{\hat\mu}_0(\Sch_X)\t K^{\hat\mu}_0(\Sch_Y)\ra
K^{\hat\mu}_0(\Sch_{X\t Y}),\;\>
\bd:\cM^{\hat\mu}_X\t\cM^{\hat\mu}_Y \ra \cM^{\hat\mu}_{X\t Y}
\end{equation*}
by following the definition above for $[S,\hat\si]\in
K^{\hat\mu}_0(\Sch_X)$ and $[T,\hat\tau]\in K^{\hat\mu}_0(\Sch_Y)$,
but taking products $S\t T$ rather than fibre products $S\t_XT$.
These $\bd$ are also commutative and associative in the appropriate
sense.

Taking $Y=\Spec\K$ and using $X\t\Spec\K\cong X$, we see that $\bd$
makes $K^{\hat\mu}_0(\Sch_X),\cM^{\hat\mu}_X$ into modules over
$K^{\hat\mu}_0(\Sch_\K),\cM^{\hat\mu}_\K$.

For generators $[S,\hat\si]$ and $[T,\hat\io]=i_X([T])$ in
$K_0^{\hat\mu}(\Sch_X)$ or $\cM^{\hat\mu}_X$ where $[T,\hat\io]$ has
trivial $\hat\mu$-action $\hat\io$, one can show that
$[S,\hat\si]\od[T,\hat\io]=[S,\hat\si]\cdot[T,\hat\io]$. Thus $i_X$
is a ring morphism $\bigl(K_0(\Sch_X),\cdot\bigr)\ra
\bigl(K_0^{\hat\mu}(\Sch_X),\od\bigr)$. However, $\Pi_X$ is not a
ring morphism $\bigl(K_0^{\hat\mu}(\Sch_X),\od\bigr)\ra
\bigl(K_0(\Sch_X),\cdot\bigr)$. Since $\bL=[\bA^1_X,\hat\io]$ this
implies that $M\cdot\bL=M\od\bL$ for all $M$ in
$K_0^{\hat\mu}(\Sch_X)$ or $\cM^{\hat\mu}_X$.
\label{mo2def3}
\end{dfn}

\begin{rem} Our principal references for this section
\cite{DeLo1,DeLo2,DeLo3,Looi} work in terms not of $\K$-schemes but
$\K$-{\it varieties}, by which they mean reduced, separated
$\K$-schemes of finite type, not necessarily irreducible.

We assume that all our schemes are of finite type unless we
explicitly say otherwise. The reduced and separated conditions in
\cite{DeLo1,DeLo2,DeLo3,Looi} are not actually needed, but are
included as an aesthetic choice, as mathematicians in some areas
prefer varieties to schemes. However, our results (Theorem
\ref{mo3thm1} for instance) concern non-reduced schemes, and would
be false if we replaced schemes by varieties, so we have
standardized on working with finite type schemes.

To see that dropping the reduced and separated conditions changes
nothing, note that if $S$ is an $X$-scheme then the reduced
$X$-subscheme $S^\red\subseteq S$ is closed with $S\sm S^\red=\es$,
so $[S]=[S^\red]$ in $K_0(\Sch_X)$. Also, any non-separated scheme
can be cut into finitely many separated schemes using the relation
$[S]=[T]+[S\sm T]$ for $T\subseteq S$ closed.
Therefore~$K_0(\Var_X)\cong K_0(\Sch_X)$.
\label{mo2rem1}
\end{rem}

\begin{ex} Define elements $\bL^{1/2}$ in
$K^{\hat\mu}_0(\Sch_X)$ and $\cM^{\hat\mu}_X$ by
\e
\bL^{1/2}=[X,\hat\io]-[X\t\mu_2,\hat\rho],
\label{mo2eq4}
\e
where $[X,\hat\io]$ with trivial $\hat\mu$-action $\hat\io$ is the
identity in $K^{\hat\mu}_0(\Sch_X),\cM^{\hat\mu}_X$, and
$X\t\mu_2=X\t\{1,-1\}$ is two copies of $X$ with nontrivial
$\hat\mu$-action $\hat\rho$ induced by the left action of $\mu_2$ on
itself, exchanging the two copies of $X$. Applying \eq{mo2eq3} with
$n=2$, we can show that $\bL^{1/2}\od \bL^{1/2}=\bL$. Thus,
$\bL^{1/2}$ in \eq{mo2eq4} is a square root for $\bL$ in the rings
$\bigl(K_0^{\hat\mu}(\Sch_X),\od\bigr)$ and
$\bigl(\cM^{\hat\mu}_X,\od\bigr)$. Note that $\bL^{1/2}\cdot
\bL^{1/2}\ne\bL$.

We can now define unique elements $\bL^{n/2}$ in
$K_0^{\hat\mu}(\Sch_X)$ for all $n=0,1,2,\ldots$ and
$\bL^{n/2}\in\cM^{\hat\mu}_X=K_0^{\hat\mu}(\Sch_X)[\bL^{-1}]$ for
all $n\in\Z$ in the obvious way, such that $\bL^{m/2}\od
\bL^{n/2}=\bL^{(m+n)/2}$ for all~$m,n$.
\label{mo2ex1}
\end{ex}

\subsection{Arc spaces and the motivic zeta function}
\label{mo22}

\begin{dfn} Let $U$ be a $\K$-scheme. For each
$n\in\N=\{0,1,\ldots\}$ we consider the {\it space\/ $\cL_n(U)$ of
arcs modulo $t^{n+1}$ on\/} $U$. This is a $\K$-scheme, whose
$K$-points, for any field $K$ containing $\K$, are the
$K[t]/t^{n+1}K[t]$-points of $U$. For $n\le m$ there are projections
$\pi_m^n:\cL_m(U)\ra\cL_n(U)$ mapping $\al\mapsto\al\mod t^{m+1}$.
Note that $\cL_0(U)=U$ and $\cL_1(U)$ is the tangent sheaf of $U$.
\label{mo2def4}
\end{dfn}

\begin{ex} If $U\subseteq\bA^m$ is a $\K$-scheme defined
by equations $f_k(x_1,\ldots,x_m)=0$ for $k=1,\ldots,l$, then
$\cL_n(U)\subseteq\bA^{m(n+1)}$ is given in the variables $x_i^j$
for $i=1,\ldots,m$ and $j=0,\ldots,n$ by the equations
\begin{equation*}
f_k\bigl(x_1^0\!+\!x_1^1t\!+\!\cdots\!+\!x_1^nt^n,\ldots,
x_m^0\!+\!x_1^1t\!+\!\cdots\!+\!x_m^nt^n\bigr)\equiv 0\!\!\!\mod
t^{n+1},\;\> k\!=\!1,\ldots,l.
\end{equation*}

\label{mo2ex2}
\end{ex}

We now recall the motivic zeta function from Denef and Loeser
\cite[\S 3.2]{DeLo3}.

\begin{dfn} Let $U$ be a smooth $\K$-scheme and $f:U\ra \bA^1$ a
non-constant regular function. Then $f$ induces morphisms
$\cL_n(f):\cL_n(U)\ra\cL_n(\bA^1)$ for $n\ge 1$. Any point $\be$ of
$\cL_n(\bA^1)$ yields a power series $\be(t)\in K[[t]]/t^{n+1}$, for
some field $K$ containing $\K$. So we define maps
\begin{equation*}
{\rm ord}_t:\cL_n(\bA^1) \ra \{ 0,1,\cdots,n,\iy\},
\end{equation*}
with ${\rm ord}_t\be$ the largest $m$ such that $t^m$ divides
$\be(t)$. Set
\begin{equation*}
\fU_n:=\bigl\{\al\in \cL_n(U):{\rm ord}_t\bigl(\cL_n(f)(\al)
\bigr)=n\bigr\}.
\end{equation*}
This is a locally closed subscheme of $\cL_n(U)$. Note that $\fU_n$
is actually a $U_0$-scheme, through the morphism $\pi^n_0:\cL_n(U)
\ra U$, where $U_0$ denotes the locus of $f=0$ in $U$. Indeed
$\pi^n_0(\fU_n) \subset U_0$, since $n\geq 1$. We consider the
morphism
\begin{equation*}
\bar f_n:\fU_n\ra\bG_m:=\bA^1\sm\{0\},
\end{equation*}
sending a point $\al$ in $\fU_n$ to the coefficient of $t^n$ in
$\cL_n(f)(\al)$. There is a natural action of $\bG_m$ on $\fU_n$ given by
$a \cdot \al(t) = \al(at)$, where $\al(t)$ is the vector of power
series corresponding to $\al$ in some local coordinate system. Since
$\bar f_n(a \cdot \al) = a^n \bar f_n(\al)$ it follows that $\bar
f_n$ is an \'etale locally trivial fibration.

We denote by $\fU_{n,1}$ the fibre $\bar f^{-1}_n(1)$. Note that the
action of $\bG_m$ on $\fU_n$ induces a good action $\rho_n$ of
$\mu_n$ (and hence a good action $\hat\rho$ of $\hat\mu$) on
$\fU_{n,1}$. Since $\bar f_n$ is a locally trivial fibration, the
$U_0$-scheme $\fU_{n,1}$ and the action of $\mu_n$ on it, completely
determines both the scheme $\fU_n$ and the morphism
\begin{equation*}
(\bar f_n,\pi^n_0) : \fU_n \longra \bG_m \t U_0.
\end{equation*}
Indeed it is easy to verify that $\fU_n$, as a $(\bG_m \t
U_0)$-scheme, is isomorphic to the quotient of $\fU_{n,1} \t \bG_m$
under the $\mu_n$-action defined by $a(\al,b)=(a \al,a^{-1}b)$.

The {\it motivic zeta function\/} $Z_f(T)\in
\cM^{\hat\mu}_{U_0}[[T]]$ of $f:U\ra\bA^1$ is the power series in
$T$ over $\cM^{\hat\mu}_{U_0}$ defined by
\e
Z_f(T):=\ts\sum\limits_{n\geq 1} \, \bigl[\fU_{n,1}\ra U_0,
\hat\rho\bigr]\,\bL^{-n\dim U}\,T^n.
\label{mo2eq5}
\e
\label{mo2def5}
\end{dfn}

We will recall a formula for $Z_f(T)$ in terms of resolution of
singularities.

\begin{dfn} Let $U$ be a smooth $\K$-scheme and $f:U\ra \bA^1$ a
non-constant regular function. By Hironaka's Theorem \cite{Hiro} we
can choose a {\it resolution\/} $(V,\pi)$ of $f$. That is, $V$ is a
smooth $\K$-scheme and $\pi:V\ra U$ a proper morphism, such that
$\pi\vert_{V\sm \pi^{-1}(U_0)}:V\sm \pi^{-1}(U_0)\ra U \sm U_0$ is
an isomorphism for $U_0=f^{-1}(0)$, and $\pi^{-1}(U_0)$ has only
normal crossings as a $\K$-subscheme of~$V$.

Write $E_i$, $i\in J$ for the irreducible components of
$\pi^{-1}(U_0)$. For each $i\in J$, denote by $N_i$ the multiplicity
of $E_i$ in the divisor of $f\ci\pi$ on $V$, and by $\nu_i - 1$ the
multiplicity of $E_i$ in the divisor of $\pi^*(\d x)$, where $\d x$
is a local non vanishing volume form at any point of $\pi(E_i)$. For
$I \subset J$, we consider the smooth $\K$-scheme~$E^\ci_I=
\bigl(\bigcap_{i \in I}E_i\bigr) \sm \bigl(\bigcup_{j \in J \sm I}
E_j\bigr)$.

Let $m_I=\gcd(N_i)_{i \in I}$. We introduce an unramified Galois
cover $\ti E^\ci_I$ of $E^\ci_I$, with Galois group $\mu_{m_I}$, as
follows. Let $V'$ be an affine Zariski open subset of $V$, such
that, on $V'$, $f \ci \pi = uv^{m_I}$, with $u:V'\ra\bA^1\sm\{0\}$
and $v:V'\ra\bA^1$. Then the restriction of $\ti E^\ci_I$ above
$E^\ci_I \cap V'$, denoted by $\ti E_I^\ci \cap V'$, is defined as
\begin{equation*}
\ti E_I^\ci \cap V'=\bigl\{(z,w)\in\bA^1\t(E^\ci_I \cap V'):
z^{m_I}=u(w)^{-1}\bigr\}.
\end{equation*}
Note that $E^\ci_I$ can be covered by such affine open subsets $V'$
of $V$. Gluing together the covers $\ti E^\ci_I \cap V'$ in the
obvious way, we obtain the cover $\ti E^\ci_I$ of $E^\ci_I$ which
has a natural $\mu_{m_I}$-action $\rho_I$, obtained by multiplying
the $z$-coordinate with the elements of $\mu_{m_I}$. This
$\mu_{m_I}$-action on $\ti E^\ci_I$ induces a $\hat\mu$-action
$\hat\rho_I$ on $\ti E^\ci_I$ in the obvious way.
\label{mo2def6}
\end{dfn}

Denef and Loeser \cite{DeLo4} and Looijenga \cite{Looi} prove that
$Z_f(T)$ is rational:

\begin{thm} In the situation of Definition\/ {\rm\ref{mo2def6},} in
$\cM^{\hat\mu}_{U_0}[[T]]$ we have:
\e
Z_f(T)=\sum_{\es \ne I \subset J} \, (\bL - 1)^{\md{I}-1} \,
\bigl[\ti E^\ci_I\ra U_0, \hat\rho_I\bigr]\, \prod_{i \in I}
\frac{\bL^{-\nu_i} T^{N_i}}{1 - \bL^{-\nu_i} T^{N_i}}\,.
\label{mo2eq6}
\e
\label{mo2thm1}
\end{thm}

\subsection{Motivic nearby and vanishing cycles}
\label{mo23}

Following Denef and Loeser \cite{DeLo3,DeLo4,DeLo1}, we state the
main definition of the section:

\begin{dfn} Let $U$ be a smooth $\K$-scheme and $f:U\ra\bA^1$ a
regular function. If $f$ is non-constant, expanding the rational
function $Z_f(T)$ as a power series in $T^{-1}$ and taking minus its
constant term, yields a well defined element of
$\cM_{U_0}^{\hat\mu}$, which we call {\it motivic nearby cycle} of
$f$. Namely,
\e
MF_{U,f}^{\rm mot}:=-\lim_{T\ra\iy}Z_f(T)=
\sum_{\es\ne I\subseteq J}(1-\bL)^{|I|-1}
\bigl[\ti E^\ci_I\ra U_0, \hat\rho_I\bigr].
\label{mo2eq7}
\e
If $f$ is constant we set $MF_{U,f}^{\rm mot}=0$. We write
\begin{equation*}
MF_{U,f}^{\rm mot}(x):={\rm Fibre}_x\bigl(MF_{U,f}^{\rm mot}\bigr)
\in \cM^{\hat\mu}_\K
\end{equation*}
for each $x\in U_0$, which we call the {\it motivic Milnor fibre} of
$f$ at~$x$.

Now let $X=\Crit(f)\subseteq U$, as a closed $\K$-subscheme of $U$,
and $X_0=X\cap U_0$. Consider the restriction $MF_{U,f}^{\rm
mot}\vert_{U_0\sm X_0}$ in $\cM_{U_0\sm X_0}^{\hat\mu}$. In
Definition \ref{mo2def6} we can choose $(V,\pi)$ with
$\pi\vert_{V\sm\pi^{-1}(X_0)}:V\sm\pi^{-1}(X_0)\ra U\sm X_0$ an
isomorphism. Write $D_1,\ldots,D_k$ for the irreducible components
of $\pi^{-1}(U_0\sm X_0)\cong U_0\sm X_0$. They are disjoint as
$\pi^{-1}(U_0\sm X_0)$ is nonsingular. The closures $\ov
D_1,\ldots,\ov D_k$ (which need not be disjoint) are among the
divisors $E_i$, so we write $\ov D_a=E_{i_a}$ for $a=1,\ldots,k$,
with $\{i_1,\ldots,i_k\}\subseteq I$. Clearly $N_{i_a}=\nu_{i_a}=1$
for $a=1,\ldots,k$.

Then in \eq{mo2eq7} the only nonzero contributions to $MF_{U,f}^{\rm
mot}\vert_{U_0\sm X_0}$ are from $I=\{i_a\}$ for $a=1,\ldots,k$,
with $\ti E^\ci_{\{i_a\}}\cong E^\ci_{\{i_a\}}\cong D_a$, and the
$\hat\mu$-action on $\ti E^\ci_{\{i_a\}}$ is trivial as it factors
through the action of $\mu_1=\{1\}$. Hence
\begin{equation*}
MF_{U,f}^{\rm mot}\vert_{U_0\sm X_0}=\ts\sum_{a=1}^k
\bigl[\ti E^\ci_{\{i_a\}},\hat\io\bigr]=\ts\sum_{a=1}^k
\bigl[D_a,\hat\io\bigr]=
\bigl[U_0\sm X_0,\hat\io\bigr].
\end{equation*}
Therefore $[U_0,\hat\io]-MF_{U,f}^{\rm mot}$ is supported on
$X_0\subseteq U_0$, and by restricting to $X_0$ we regard it as an
element of~$\cM_{X_0}^{\hat\mu}$.

As $f\vert_X:X\ra\bA^1$ is locally constant on
$X^\red=\Crit(f)^\red$, $f(X)$ is finite, and $X=\coprod_{c\in
f(X)}X_c$ with $X_c=X\cap U_c$, where $U_c=f^{-1}(c)\subset U$.
Define $MF_{U,f}^{\rm mot,\phi}\in\cM_X^{\hat\mu}$ by
\e
MF_{U,f}^{\rm mot,\phi}\big\vert_{X_c}=\bL^{-\dim U/2}\od
\bigl([U_c,\hat\io]-MF_{U,f-c}^{\rm mot}\bigr)\big\vert_{X_c}\in
\cM^{\hat\mu}_{X_c}
\label{mo2eq8}
\e
for each $c\in f(X)$, where $\bL^{-\dim U/2}\in\cM^{\hat\mu}_{X_c}$
is defined in Example \ref{mo2ex1}, and the product $\od$ in
Definition \ref{mo2def3}. We call $MF_{U,f}^{\rm mot,\phi}$ the {\it
motivic vanishing cycle\/} of $f$.

Sometimes it is convenient to regard $MF_{U,f}^{\rm mot,\phi}$ as an
element of $\cM_U^{\hat\mu}$ supported on $X$, via the inclusion
$\cM_X^{\hat\mu}\hookra\cM_U^{\hat\mu}$. Also, for each $x\in X$ we
set
\e
MF_{U,f}^{\rm mot,\phi}(x):= \bL^{-\dim U/2}\od\bigl(1-MF_{U,f}^{\rm
mot}(x)\bigr) \in \cM^{\hat\mu}_\K.
\label{mo2eq9}
\e

In \S\ref{mo3}--\S\ref{mo4} we will use the fact that if $U,V$ are
smooth $\K$-schemes, $f:U\ra\bA^1$, $g:V\ra\bA^1$ are regular,
$X=\Crit(f)$, $Y=\Crit(g)$, and $\Phi:U\ra V$ is \'etale with
$f=g\ci\Phi$, so that $\Phi\vert_X:X\ra Y$ is \'etale, then
$\Phi\vert_X^*\bigl(MF_{V,g}^{\rm mot,\phi}\bigr)=MF_{U,f}^{\rm
mot,\phi}$.
\label{mo2def7}
\end{dfn}

\begin{rem}{\bf(a)} Because of \eq{mo2eq6} the right hand side of
\eq{mo2eq7} is independent of the choice of resolution $(V,\pi)$,
although a priori this is not at all obvious.
\smallskip

\noindent{\bf(b)} As in \cite[\S 3.5]{DeLo3}, one should regard
$MF_{U,f}^{\rm mot}(x)$ as the correct motivic incarnation of the
{\it Milnor fibre\/} $MF_{U,f}(x)$ of $f$ at $x$ when $\K=\C$, which
is in itself not at all motivic. This is indeed true for the Hodge
realization \cite[Th.~3.10]{DeLo3}. Moreover, using the perverse
sheaf notation of Brav et al.\ \cite[\S 2]{BBDJS}, $MF_{U,f}^{\rm
mot}$ is the virtual motivic incarnation of the {\it complex of
nearby cycles\/} $\psi_f(\Q_U)$ of $U,f$ in $D^b_c(U_0)$, and
$MF_{U,f}^{\rm mot,\phi}$ as the virtual motivic incarnation of the
{\it perverse sheaf of vanishing cycles\/}
$\PV^\bu_{U,f}\in\Perv(X)$ of $U,f$, defined by
\e
\PV^\bu_{U,f}\vert_{X_c}=\phi_{f-c}[-1]\bigl(\Q_U[\dim U]\bigr)\quad
\text{for all $c\in f(X)$.}
\label{mo2eq10}
\e
Note the analogy between \eq{mo2eq8} and \eq{mo2eq10}.
\smallskip

\noindent{\bf(c)} Our formulae \eq{mo2eq8}--\eq{mo2eq9} are based on
Denef and Loeser\cite[Not.~3.9]{DeLo3}, but with a different
normalization, since Denef and Loeser have $(-1)^{\dim U}$ rather
than $\bL{}^{-\dim U/2}$. As in Examples \ref{mo2ex3} and
\ref{mo2ex4} below, our normalization ensures that $MF^{\rm
mot,\phi}_{V,q}=1$ whenever $q$ is a nondegenerate quadratic form on
a finite-dimensional $\K$-vector space $V$. The normalizing factor
$\bL{}^{-\dim U/2}$ also appears in Kontsevich and
Soibelman~\cite[\S 5.1]{KoSo1}.
\label{mo2rem2}
\end{rem}

\begin{ex} Define $f:\bA^1\ra\bA^1$ by $f(z)=z^2$. In
Definition \ref{mo2def6} we may take $U=V=\bA^1$ and
$\pi=\id_{\bA^1}$. Then $\pi^{-1}(0)=\{0\}$ is one divisor
$E_0=\{0\}$, with $N_0=2$ and $\nu_0=1$. In \eq{mo2eq6} the only
nonzero term is $I=\{0\}$, and $\ti E^\ci_{\{0\}}=\mu_2=\{1,-1\}$ is
two points with $\hat\mu$-action $\hat\rho$ induced by the left
action of $\mu_2$ on itself. Hence
$Z_f(T)=[\mu_2,\hat\rho]\cdot(\bL^{-1} T^2)/(1-\bL^{-1}T^2)$. Taking
the limit $T\ra\iy$, equation \eq{mo2eq7} yields
$MF_{\bA^1,z^2}^{\rm mot}=[\mu_2,\hat\rho]$. Thus \eq{mo2eq4} and
\eq{mo2eq8} give
\e
MF_{\bA^1,z^2}^{\rm mot,\phi}=\bL^{-1/2}\od
\bigl(1-[\mu_2,\hat\rho]\bigr)= \bL^{-1/2}\od\bL^{1/2}=1.
\label{mo2eq11}
\e

\label{mo2ex3}
\end{ex}

\subsection{The motivic Thom--Sebastiani Theorem}
\label{mo24}

Here is the motivic Thom--Sebastiani Theorem of Denef--Loeser and
Looijenga \cite{DeLo3,DeLo2,Looi}, stated using the notation of
\S\ref{mo21}--\S\ref{mo23}.

\begin{thm} Let\/ $U,V$ be smooth\/ $\K$-schemes, $f:U\ra\bA^1,$
$g:V\ra \bA^1$ regular functions, and\/ $X=\Crit(f),$ $Y=\Crit(g)$.
Write $f\boxplus g:U\t V\ra\bA^1$ for the regular function mapping
$f\boxplus g:(u,v)\mapsto f(u)+g(v)$. Then $MF_{U\t V,f \boxplus
g}^{\rm mot,\phi}=MF_{U,f}^{\rm mot,\phi}\bd MF_{V,g}^{\rm
mot,\phi}$ in $\cM^{\hat\mu}_{X\t Y}$.
\label{mo2thm2}
\end{thm}

\begin{ex} Define $f:\bA^n\ra\bA^1$ by
$f(z_1,\ldots,z_n)=z_1^2+\cdots+z_n^2$ for $n\ge 1$. Then using
Theorem \ref{mo2thm2}, induction on $n$, and equation \eq{mo2eq11}
shows that
\e
MF^{\rm mot,\phi}_{\bA^n,z_1^2+\cdots+z_n^2}=MF_{\bA^1,z^2}^{\rm
mot,\phi}\bd\cdots\bd MF_{\bA^1,z^2}^{\rm mot,\phi}=1\bd\cdots\bd
1=1.
\label{mo2eq12}
\e
If $V$ is a finite-dimensional $\K$-vector space and $q$ a
nondegenerate quadratic form on $V$, then
$(V,q)\cong(\bA^n,z_1^2+\cdots+z_n^2)$ for $n=\dim V$, so $MF^{\rm
mot,\phi}_{V,q}=1$. The purpose of the factors $\bL^{-\dim U/2}$ in
\eq{mo2eq8}--\eq{mo2eq9} was to achieve this.
\label{mo2ex4}
\end{ex}

Our next result shows how motivic vanishing cycles change under
stabilization by a nondegenerate quadratic form. The term $\bL^{\dim
U/2}\od MF_{E,q}^{\rm mot,\phi}$ in \eq{mo2eq13} may be regarded as
the {\it relative motivic vanishing cycle\/} of $(E,q)$ relative
to~$U$.

\begin{thm} Let\/ $U$ be a smooth\/ $\K$-scheme, $\pi:E\ra U$ a
vector bundle over $U,$ $f:U\ra\bA^1$ a regular function, $q$ a
nondegenerate quadratic form on $E,$ and\/ $X=\Crit(f)$. Regard (the
total space of) $E$ as a smooth\/ $\K$-scheme and\/
$q,f\ci\pi:E\ra\bA^1$ as regular functions on $E,$ so that
$f\ci\pi+q:E\ra\bA^1$ is also a regular function. Identify $U$ with
the zero section in $E,$ so that\/ $X\subseteq U\subseteq E,$ and we
have $\cM^{\hat\mu}_X\subseteq \cM^{\hat\mu}_U\subseteq
\cM^{\hat\mu}_E$. Then in $\cM^{\hat\mu}_E$ we have
\e
MF_{E,f\ci\pi+q}^{\rm mot,\phi}=MF_{U,f }^{\rm mot,\phi}\od
\bigl(\bL^{\dim U/2}\od MF_{E,q}^{\rm mot,\phi}\bigr).
\label{mo2eq13}
\e
\label{mo2thm3}
\end{thm}

\begin{proof} Consider the function $f\boxplus q:U\t
E\ra\bA^1$ and morphism $\pi\t\id:E\ra U\t E$. We have $(f\boxplus
q)\ci(\pi\t\id)=f\ci\pi+q:E\ra\bA^1$. We first claim that
\e
(\pi\t\id)^*\bigl(Z_{f\boxplus q}(T)\bigr)=Z_{f\ci\pi+q}(T)\in
\cM^{\hat\mu}_E[[T]].
\label{mo2eq14}
\e
To see this, note that \'etale locally on $U\t E$, there exist
isomorphisms $U\t E\cong E\t\bA^m$, where $m=\dim U$, making the
following two diagrams equivalent:
\e
\begin{gathered}
\xymatrix@C=8pt@R=17pt{ *+[r]{E} \ar[d]^{\pi\t\id}
\ar@/^/[drrrrrrr]^{f\ci\pi+q} &&&&&&&& *+[r]{E}
\ar[d]^{\id\t 0} \ar@/^/[drrrrrrr]^{f\ci\pi+q} \\
*+[r]{U\t E} \ar[rrrrrrr]^(0.4){f\boxplus q} &&&&&&&
*+[r]{\bA^1,} & *+[r]{E\t\bA^m}
\ar[rrrrrrr]^(0.5){(f\ci\pi+q)\boxplus 0} &&&&&&& *+[r]{\bA^1.} }
\end{gathered}
\label{mo2eq15}
\e

It follows that in Definition \ref{mo2def5}, we have \'etale local
isomorphisms
\e
\begin{split}
(\pi\t\id)^*\bigl((\fU\t\fE)_n\bigr)&\cong
\fE_n\t\cL_n(\bA^m)_0,\\
(\pi\t\id)^*\bigl((\fU\t\fE)_{n,1}\bigr)&\cong
\fE_{n,1}\t\cL_n(\bA^m)_0,
\end{split}
\label{mo2eq16}
\e
where $\fE_n,\fE_{n,1}$ and $(\fU\t\fE)_n,(\fU\t\fE)_{n,1}$ are
$\fU_n,\fU_{n,1}$ in Definition \ref{mo2def5} for
$f\ci\pi+q:E\ra\bA^1$ and $f\boxplus q:U\t E\ra\bA^1$, and
$\cL_n(\bA^m)_0$ is the subspace of arcs in $\cL_n(\bA^m)$ based at
0. The second equation also holds with $\hat\mu$-actions, where
$\cL_n(\bA^m)_0$ has the trivial $\hat\mu$-action~$\hat\io$.

We claim that the isomorphisms \eq{mo2eq16} also hold
Zariski locally, and thus on the level of motives. To see this, note that we may cover $U$ by Zariski open neighbourhoods $U'\subseteq U$, such that on $U'$ we can choose \'etale coordinates $(z_1,\ldots,z_m)$, and writing $E'=E\vert_{U'}$ and $q'=q\vert_{U'}$, we can choose an algebraic connection $\nabla$ on $E'$ which preserves $q'$. Since \eq{mo2eq16} over $E'$ concerns formal arcs starting in the images $(\pi\t\id)(E')$ and $E'\t 0$ in $U'\t E'$ and $E'\t\bA^m$, it depends only on the formal completions $\widehat{(U'\t E')}_{(\pi\t\id)(E')}$ and $\widehat{(E'\t\bA^m)}_{E'\t 0}$. 

Informally, we think of points of $\widehat{(U'\t E')}_{(\pi\t\id)(E')}$ as pairs $(u_0,(u_1,e_1))$ where $u_0,u_1\in U'$, and $e_1\in E'\vert_{u_1}$, and $u_0,u_1$ are `infinitesimally close' in $U'$, and $f\boxplus q$ maps $(u_0,(u_1,e_1))\mapsto f(u_0)+q'\vert_{u_1}(e_1)$. Similarly, we think of points of $\widehat{(E'\t\bA^m)}_{E'\t 0}$ as pairs $((u,e),x)$, where $u\in U'$, and $e\in E'\vert_u$, and $x\in\bA^m$ is `infinitesimally close' to 0, and $
(f\ci\pi+q)\boxplus 0$ maps~$((u,e),x)\mapsto f(u)+q'\vert_u(e)$.

We can now define a unique isomorphism of formal schemes
\begin{equation*}
\Phi:\widehat{(U'\t E')}_{(\pi\t\id)(E')}\,{\buildrel\cong\over\longra}\,\widehat{(E'\t\bA^m)}_{E'\t 0}
\end{equation*}
which (informally) on points maps $\Phi:(u_0,(u_1,e_1))\mapsto ((u,e),x)$, where $u=u_0$, and $x=\bigl(z_1(u_1)-z_1(u_0),\ldots,z_m(u_1)-z_m(u_0)\bigr)$, and $e\in E'\vert_{u_0}$ is the unique point such that if $\phi:\widehat{(\bA^m)}_0\ra U'$ is the unique morphism of formal schemes with $\phi(0)=u_0$, $\phi(x)=u_1$ and $z_i\ci\phi(y_1,\ldots,y_m)=z_i(u_0)+y_i$ for $i=1,\ldots,m$, then parallel translation from $t=0$ to $t=1$ along the formal path $t\mapsto \phi(t\cdot x)$ in $U'$ for $t\in\bA^1$ in the vector bundle $E'\ra U'$ using the connection $\nabla$ maps $e\in E'\vert_{u_0}$ at $t=0$ to $e_1\in E'\vert_{u_1}$ at $t=1$. As $\nabla$ preserves $q'$ we have $q'\vert_{u_0}(e)=q'\vert_{u_1}(e_1)$, so that $f(u_0)+q'\vert_{u_1}(e_1)=f(u)+q'\vert_u(e)$. Thus $\Phi$ is compatible with the commutative triangles \eq{mo2eq15}, and induces Zariski local isomorphisms \eq{mo2eq16} over $U'$, proving the claim. 

Since $[\cL_n(\bA^m)_0,\hat\io]=\bL^{mn}$, equation \eq{mo2eq16} at the level of motives gives
\begin{equation*}
(\pi\t\id)^*\bigl[(\fU\t\fE)_{n,1},\hat\rho\bigr]=
\bigl[\fE_{n,1},\hat\rho\bigr]\cdot\bL^{mn}.
\end{equation*}
Multiplying this by $\bL^{-n(\dim E +m)}T^n$, summing over all $n\ge
1$, and using \eq{mo2eq5}, proves \eq{mo2eq14}. Taking the limit
$T\ra\iy$ in \eq{mo2eq14} and using \eq{mo2eq7} yields
\e
(\pi\t\id)^*\bigl(MF_{U\t E,f\boxplus q}^{\rm mot}\bigr)=
MF_{E,f\ci\pi+q}^{\rm mot}.
\label{mo2eq17}
\e

Now both sides of \eq{mo2eq13} are supported on $X\subseteq
U\subseteq E$, and as in Definition \ref{mo2def7} $X=\coprod_{c\in
f(X)}X_c$ with $X_c=X\cap f^{-1}(c)$. For each $c\in f(X)$ we have
\begin{align*}
MF_{E,f\ci\pi+q}^{\rm mot,\phi}\vert_{X_c}&=
\bL^{-\dim E/2}\od \bigl([E_c,\hat\io]-MF_{E,(f-c)\ci\pi+q}^{\rm
mot}\bigr)\big\vert_{X_c}\\
&=\bL^{\dim U/2}\od (\pi\t\id)\vert_{X_c}^*\bigl[
\bL^{-\dim U/2-\dim E/2}\od\\
&\qquad\qquad\bigl([(U\t E)_c,\hat\io]
-MF_{U\t E,(f-c)\boxplus q}^{\rm
mot}\bigr)\big\vert_{X_c\t U}\bigr]\\
&=\bL^{\dim U/2}\od (\pi\t\id)\vert_{X_c}^*\bigl[MF_{U\t E,f\boxplus
q}^{\rm mot,\phi}\big\vert_{X_c\t U}\bigr]\\
&=\bL^{\dim U/2}\od (\pi\t\id)\vert_{X_c}^*\bigl[MF_{U,f }^{\rm mot,\phi}
\bd MF_{E,q}^{\rm mot,\phi}\big\vert_{X_c\t U}\bigr]\\
&=MF_{U,f }^{\rm mot,\phi} \od
\bigl(\bL^{\dim U/2}\od MF_{E,q}^{\rm mot,\phi}\bigr)\big\vert_{X_c},
\end{align*}
using \eq{mo2eq8} in the first and third steps, \eq{mo2eq17} with
$f-c$ in place of $f$ in the second, Theorem \ref{mo2thm2} in the
fourth, and comparing the definitions of $\od$ and $\bd$ in
Definition \ref{mo2def3} in the fifth: $MF_{U,f}^{\rm mot,\phi}\od
MF_{E,q}^{\rm mot,\phi}$ involves a fibre product (over $X$ or $U$
or $E$, all have the same effect), but $MF_{U,f }^{\rm mot,\phi} \bd
MF_{E,q}^{\rm mot,\phi}$ has no fibre product, and the effect of
$(\pi\t\id)^*$ is to take the fibre product. This proves the
restriction of \eq{mo2eq13} to $X_c$ for each $c\in f(X)$, and the
theorem follows.
\end{proof}

\subsection{\texorpdfstring{Motives of principal $\Z_2$-bundles}{Motives of principal ℤ₂-bundles}}
\label{mo25}

We define principal $\Z_2$-bundles $P\ra X$, associated motives
$\Up(P)$, and a quotient ring of motives $\oM_X^{\hat\mu}$ in
which~$\Up(P\ot_{\Z_2}Q)=\Up(P)\od\Up(Q)$ for all $P,Q$.

\begin{dfn} Let $X$ be a $\K$-scheme. A {\it principal\/
$\Z_2$-bundle\/} $P\ra X$ is a proper, surjective, \'etale morphism
of $\K$-schemes $\pi:P\ra X$ together with a free involution
$\si:P\ra P$, such that the orbits of $\Z_2=\{1,\si\}$ are the
fibres of $\pi$. The {\it trivial\/ $\Z_2$-bundle\/} is
$\pi_X:X\t\Z_2\ra X$. We will use the ideas of {\it isomorphism\/}
of principal bundles $\io:P\ra Q$, {\it section\/} $s:X\ra P$, {\it
tensor product\/} $P\ot_{\Z_2}Q$, and {\it pullback\/} $f^*(P)\ra W$
under a morphism of $\K$-schemes $f:W\ra X$, all of which are
defined in the obvious ways.

Write $\Z_2(X)$ for the abelian group of isomorphism classes $[P]$
of principal $\Z_2$-bundles $P\ra X$, with multiplication
$[P]\cdot[Q]=[P\ot_{\Z_2}Q]$ and identity $[X\t\Z_2]$. Since
$P\ot_{\Z_2}P\cong X\t\Z_2$ for each $P\ra X$, each element of
$\Z_2(X)$ is self-inverse, and has order 1 or 2.

If $P\ra X$ is a principal $\Z_2$-bundle over $X$, define a motive
\e
\Up(P)=\bL^{-1/2}\od\bigl([X,\hat\io]-[P,\hat\rho]\bigr)\in
\cM_X^{\hat\mu},
\label{mo2eq18}
\e
where $\hat\rho$ is the $\hat\mu$-action on $P$ induced by the
$\mu_2$-action on $P$ from the principal $\Z_2$-bundle structure, as
$\mu_2\cong\Z_2$.  If $P=X\t\Z_2$ is the trivial $\Z_2$-bundle then
\begin{equation*}
\Up(X\t\Z_2)=\bL^{-1/2}\od\bigl([X,\hat\io]-[X\t\Z_2,\hat\rho]
\bigr)=\bL^{-1/2}\od\bL^{1/2}\od[X,\hat\io]=[X,\hat\io],
\end{equation*}
using \eq{mo2eq4}. Note that $[X,\hat\io]$ is the identity in the
ring~$\cM_X^{\hat\mu}$.

As $\Up(P)$ only depends on $P$ up to isomorphism, $\Up$ factors
through $\Z_2(X)$, and we may consider $\Up$ as a
map~$\Z_2(X)\ra\cM_X^{\hat\mu}$.

For our applications in \S\ref{mo4}--\S\ref{mo5} we want
$\Up:\Z_2(X)\ra\cM_X^{\hat\mu}$ to be a group morphism with respect
to the multiplication $\od$ on $\cM_X^{\hat\mu}$, but we cannot
prove that it is. Our (somewhat crude) solution is to pass to a
quotient ring $\oM_X^{\hat\mu}$ of $\cM_X^{\hat\mu}$ such that the
induced map $\Up:\Z_2(X)\ra\oM_X^{\hat\mu}$ is a group morphism.

In fact we want more than this: if we simply define $\oM_X^{\hat\mu}$ to be the quotient ring of $\cM_X^{\hat\mu}$ by the relations
$\Up(P\ot_{\Z_2}Q)-\Up(P)\od\Up(Q)=0$ for all $[P],[Q]$ in $\Z_2(X)$ then pushforwards $\phi_*:\oM_X^{\hat\mu}\ra\oM_Y^{\hat\mu}$ will not be defined for general morphisms $\phi:X\ra Y$. The proof of Theorem \ref{mo5thm5} below implicitly uses pushforwards $\phi_*$ for Zariski open inclusions $\phi:R\hookra X$, and for the generalization to stacks \cite[Prop.~5.8, Th.~5.14]{BBBJ} we will need pushforwards $\phi_*$ for smooth $\phi:X\ra Y$. We will define $\oM_X^{\hat\mu}$ so that all pushforwards exist.

So, for each $\K$-scheme $Y$, define $I_Y^{\hat\mu}$ to be the ideal
in the commutative ring $\bigl(\cM_Y^{\hat\mu},\od\bigr)$ generated
by elements $\phi_*\bigl(\Up(P\ot_{\Z_2}Q)-\Up(P)\od\Up(Q)\bigr)$ for all $\K$-scheme morphisms $\phi:X\ra Y$ and principal $\Z_2$-bundles $P,Q\ra X$, and define $\oM_Y^{\hat\mu}=\cM_Y^{\hat\mu}/I_Y^{\hat\mu}$ to be the quotient, as a commutative ring with multiplication `$\od$', with
projection~$\Pi_Y^{\hat\mu}:\cM_Y^{\hat\mu}\ra\oM_Y^{\hat\mu}$.

Note that in $\oM_Y^{\hat\mu}$ we do not have the second
multiplication `$\,\cdot\,$', since we do not require $I_Y^{\hat\mu}$ to be an ideal in $\bigl(\cM_Y^{\hat\mu},\cdot\bigr)$. Also the external product $\boxt$ and projection $\Pi_Y:\cM_Y^{\hat\mu}\ra\cM_Y$ on $\cM_Y^{\hat\mu}$ do not descend to $\oM_Y^{\hat\mu}$. Apart from these, all of
\S\ref{mo21}--\S\ref{mo24}, in particular the operations $\od,\bd$, pushforwards $\phi_*$ and pullbacks $\phi^*$, elements $\bL,\bL^{1/2}, MF_{U,f}^{\rm mot},\ab MF_{U,f}^{\rm mot,\phi},\ab\Up(P)$, and Theorems \ref{mo2thm1}, \ref{mo2thm2} and \ref{mo2thm3}, make sense in $\oM_Y^{\hat\mu}$ rather than $\cM_Y^{\hat\mu}$ by applying $\Pi_Y^{\hat\mu}$. We will use the same notation in $\oM_Y^{\hat\mu}$ as in~$\cM_Y^{\hat\mu}$.

Taking $Y=X$ and $\phi=\id_X$, we see that $\oM_X^{\hat\mu}$ has the property that
\e
\Up(P\ot_{\Z_2}Q)=\Up(P)\od\Up(Q)\quad\text{in $\oM_X^{\hat\mu}$}
\label{mo2eq19}
\e
for all principal $\Z_2$-bundles~$P,Q\ra X$.
\label{mo2def8}
\end{dfn}

\begin{rem}{\bf(a)} When we define a ring $R$ by generators and
relations, such as $K_0(\Sch_X),K_0^{\hat\mu}(\Sch_X),
\cM_X^{\hat\mu},\oM_X^{\hat\mu}$, and we impose an apparently
arbitrary relation, such as Definition \ref{mo2def2}(iii), or the
quotient by $I_X^{\hat\mu}$ in Definition \ref{mo2def8}, then there
is a risk that $R$ may be small or even zero. If so, theorems we
prove in $R$ will be of little or no value. Thus, when we make such
a definition, we should justify that $R$ is `reasonably large', for
instance by producing morphisms from $R$ to other interesting rings.
We do this in {\bf(b)\rm,\bf(c)}.
\smallskip

\noindent{\bf(b)} Our definition of $\oM_X^{\hat\mu}$ is based on
the motivic rings of Kontsevich and Soibelman \cite{KoSo1}. In
\cite[\S 4.3]{KoSo1} they discuss motivic rings ${\rm Mot}^\mu(X)$
which coincide with our $K_0^{\hat\mu}(\Sch_X)$, as in Denef and
Loeser \cite{DeLo3}. Then in \cite[\S 4.5]{KoSo1} they define
motivic rings $\oM^\mu(X)$ as the quotient of ${\rm Mot}^\mu(X)$ by
a complicated relation involving cohomological equivalence of
algebraic cycles, and inverting $\bL$. As in \cite[\S 5.1]{KoSo1},
this equivalence relation implies that the analogue of Theorem
\ref{mo2thm4} holds in their $\oM^\mu(X)[\bL^{-1}]$, which implies
that \eq{mo2eq19} also holds in their $\oM^\mu(X)[\bL^{-1}]$. It
follows that their $\oM^\mu(X)[\bL^{-1}]$ is a quotient of our
$\oM_X^{\hat\mu}$.
\smallskip

\noindent{\bf(c)} As in Kontsevich and Soibelman \cite{KoSo1,KoSo2},
a major goal in Donaldson--Thomas theory is {\it categorification},
i.e.\ the replacement of $K_0(\Sch_X),K_0^{\hat\mu}(\Sch_X)$ by the
Grothendieck groups of suitable categories of motives. To each
$\K$-scheme $X$ one should associate triangulated $\Q$-linear tensor
categories $\DM_X^\Q,\DM_X^{\hat\mu,\Q}$ with functors
$M_X:\Sch_X\ra\DM_X^\Q$,
$M_X^{\hat\mu}:\Sch_X^{\hat\mu}\ra\DM_X^{\hat\mu,\Q}$ satisfying a
package of properties we will not discuss. Brav et al.\ \cite{BBDJS}
prove categorified versions of the results of this paper, in the
contexts of $\Q$-linear perverse sheaves, $\scr D$-modules, and
mixed Hodge modules. A kind of universal categorification may be
given by Voevodsky's category of motives~\cite{CiDe,Voev}.

Using similar arguments to \cite[\S 4.5, \S 5.1]{KoSo1}, involving
the triangulated $\Q$-linear properties of $\DM_X^{\hat\mu,\Q}$ in
an essential way, one can show that for any such categorification
with the right properties, an analogue of Theorem \ref{mo2thm4}
below holds in $K_0(\DM_X^{\hat\mu,\Q})[\bL^{-1}]$, so
$K_0(M_X^{\hat\mu}):K_0(\Sch_X^{\hat\mu})[\bL^{-1}]\ra
K_0(\DM_X^{\hat\mu,\Q})[\bL^{-1}]$ factors via our ring
$\oM_X^{\hat\mu}$, giving a morphism $\oM_X^{\hat\mu}\ra
K_0(\DM_X^{\hat\mu,\Q})[\bL^{-1}]$.
\label{mo2rem3}
\end{rem}

Now there is a natural 1-1 correspondence
\e
\begin{split}
\Z_2(X)=\,&\bigl\{\text{isomorphism classes $[P]$ of principal
$\Z_2$-bundles $P\ra X$}\bigr\} \\
\longleftrightarrow\, &\bigl\{\text{isomorphism classes $[L,\io]$ of
pairs $(L,\io)$, where $L\ra X$ is}\\
&\quad\text{a line bundle and $\io:L\ot_{\O_X}L\ra\O_X$ an
isomorphism}\bigr\},
\end{split}
\label{mo2eq20}
\e
defined as follows: to each principal $\Z_2$-bundle $\pi:P\ra X$ we
associate the line bundle $L=(P\t\bA^1)/\Z_2$ over $X$ (identifying
line bundles with their total spaces), where $\Z_2$ acts in the
given way on $P$ and as $z\mapsto -z$ on $\bA^1$. The isomorphism
$\io:L\ot_{\O_X}L\ra\O_X$ acts on points by
$\io\bigl(((p,z)\Z_2)\ot((p,z')\Z_2)\bigr)=(x,zz')\in
X\t\bA^1=\O_X$, where $p\in P$ with $\pi(p)=x$ and $z,z'\in\bA^1$.
Conversely, to each line bundle $\la:L\ra X$ with isomorphism
$\io:L\ot_{\O_X}L\ra\O_X$, we define $P$ to be the $\K$-subscheme of
points $l\in L$ with $\io(l\ot l)=(x,1)$ for $x=\la(l)$, with
projection $\pi=\la\vert_P:P\ra X$ and $\Z_2$-action $\si:l\mapsto
-1\cdot l$.

For smooth $U$, we can express $\Up(P)$ in \eq{mo2eq18} in terms of
the motivic vanishing cycle associated to the corresponding
$(L,\io)$ in~\eq{mo2eq20}.

\begin{lem} Let\/ $U$ be a smooth\/ $\K$-scheme, $P\ra U$ a
principal\/ $\Z_2$-bundle, and\/ $(L,\io)$ correspond to $P$ under
the $1$-$1$ correspondence \eq{mo2eq20}. Define a regular function
$q:L\ra\bA^1$ on the total space of\/ $L$ by $\io(l\ot
l)=(x,q(l))\in\O_X\cong X\t\bA^1$ for $l\in L,$ so that\/ $q$ is a
nondegenerate quadratic form on the fibres of\/ $L\ra X,$ and\/
$\Crit(q)\subset L$ is the zero section of\/ $L,$ which we identify
with\/ $U$. Then
\e
\Up(P)=\bL^{\dim U/2}\od MF_{L,q}^{\rm mot,\phi}\quad \text{in
both\/ $\cM_U^{\hat\mu}$ and\/ $\oM_U^{\hat\mu}$.}
\label{mo2eq21}
\e

\label{mo2lem1}
\end{lem}

\begin{proof} By a very similar proof to Example \ref{mo2ex3}, we
may show that $MF_{L,q}^{\rm mot}=[P,\hat\rho]$, for $\hat\rho$ as
in \eq{mo2eq18}. Equation \eq{mo2eq21} then follows from
\eq{mo2eq8}, \eq{mo2eq18} and $\dim L=\dim U+1$.
\end{proof}

We can generalize \eq{mo2eq21} to an expression \eq{mo2eq22} for
motivic vanishing cycles $MF_{E,q}^{\rm mot,\phi}$ of nondegenerate
quadratic forms on vector bundles. The proof uses \eq{mo2eq19}, and
so holds only in $\oM_U^{\hat\mu}$ rather than in $\cM_U^{\hat\mu}$.
Note that $\bL^{\dim U/2}\od MF_{E,q}^{\rm mot,\phi}$ in
\eq{mo2eq22} also occurs in equation \eq{mo2eq13} of
Theorem~\ref{mo2thm3}.

\begin{thm} Let\/ $U$ be a smooth\/ $\K$-scheme, $E\ra U$ a
vector bundle of rank\/ $r,$ and\/ $q\in H^0(S^2E^*)$ a
nondegenerate quadratic form on the fibres of\/ $E$. Regard\/
$q:E\ra\bA^1$ as a regular function on the total space of\/ $E,$
which is a nondegenerate homogeneous quadratic polynomial on each
fibre $E_u$ of\/ $E,$ so that\/ $\Crit(q)\subseteq E$ is the zero
section of\/ $E,$ which we identify with\/~$U$.

Then $\La^rE\ra U$ is a line bundle, and the determinant\/ $\det(q)$
is a nonvanishing section of\/ $(\La^rE^*)^{\ot^2},$ or equivalently
an isomorphism $(\La^rE)\ot_{\O_U}(\La^rE)\ra\O_U$. Thus there is a
principal\/ $\Z_2$-bundle $P\ra U,$ unique up to isomorphism,
corresponding to $\bigl(\La^rE,\det(q)\bigr)$ under the $1$-$1$
correspondence \eq{mo2eq20}. We have
\e
\Up(P)=\bL^{\dim U/2}\od MF_{E,q}^{\rm mot,\phi}\quad \text{in\/
$\oM_U^{\hat\mu}$.}
\label{mo2eq22}
\e

\label{mo2thm4}
\end{thm}

\begin{proof} We will prove the theorem by induction on $r=\rank E$.
The first step, with $r=1$, follows from Lemma \ref{mo2lem1}. For
the inductive step, suppose the theorem holds for all $U,E,q$ with
$\rank E=r\le n$ for $n\ge 1$, and let $U,E,q$ be as in the theorem
with $\rank E=n+1$. It is enough to prove \eq{mo2eq22} Zariski
locally on $U$. For each $x\in U$, we can choose a Zariski open
neighbourhood $U'$ of $x$ in $U$ and a section $s\in
H^0(E\vert_{U'})$ such that $q(s,s)$ is nonvanishing on $U'$.

Write $E\vert_{U'}=F\op L$, where $L=\langle s\rangle$ is the line
subbundle of $E$ spanned by $s$, and $F=L^\perp$ the orthogonal
vector subbundle of $L$ in $E$ with respect to $q$. This makes sense
as $q(s,s)\ne 0$ on $U'$, so $L^\perp\cap L=\{0\}$. Then
$q\vert_{U'}=q_F\op q_L$, for $q_F\in H^0(S^2F^*)$ and $q_L\in
H^0(S^2L^*)$ nondegenerate quadratic forms on $F,L$.

Using the projection $E\vert_{U'}=F\op L\ra F$, we can regard the
total space of $E\vert_{U'}$ as a line bundle over the total space
of $F$, with $E\vert_{U'}\cong \pi^*(L)$ for $\pi:F\ra U'$ the
projection. Then Theorem \ref{mo2thm3} with $F,E\vert_{U'},q_F,q_L$
in place of $U,E,f,q$ gives
\e
MF_{E\vert_{U'},q\vert_{U'}}^{\rm mot,\phi}=MF_{F,q_F}^{\rm
mot,\phi}\od \bigl(\bL^{\dim F/2}\od MF_{E\vert_{U'},q_L}^{\rm
mot,\phi}\bigr).
\label{mo2eq23}
\e

As $\rank F=n$, the inductive hypothesis gives
\e
\Up(Q)=\bL^{\dim U/2}\od MF_{F,q_F}^{\rm mot,\phi},
\label{mo2eq24}
\e
where $Q\ra U'$ is the principal $\Z_2$-bundle corresponding to
$\bigl(\La^rF,\det(q_F)\bigr)$ under \eq{mo2eq20}. Also Lemma
\ref{mo2lem1} gives
\e
\Up(R)=\bL^{\dim F/2}\od MF_{E\vert_{U'},q_L}^{\rm mot,\phi},
\label{mo2eq25}
\e
with $R\ra F$ the principal $\Z_2$-bundle corresponding to
$\bigl(E\vert_{U'}\cong \pi^*(L),q_L\bigr)$.

We now have
\e
\begin{split}
&MF_{E\vert_{U'},q\vert_{U'}}^{\rm mot,\phi}=\bigl(\bL^{-\dim U/2}
\od\Up(Q)\bigr)\od\Up(R)\\
&\quad=\bL^{-\dim U/2}\od\bigl(\Up(Q)\od\Up(R\vert_{U'})\bigr)=
\bL^{-\dim U/2}\od\Up(P\vert_{U'}),
\end{split}
\label{mo2eq26}
\e
where we consider that $U'\subseteq F\subseteq E\vert_{U'}$, and
regard \eq{mo2eq23}--\eq{mo2eq26} as equations in
$\oM_{E\vert_{U'}}^{\hat\mu}\supseteq\oM_F^{\hat\mu}
\supseteq\oM_{U'}^{\hat\mu}$ with $MF_{E\vert_{U'},q\vert_{U'}}^{\rm
mot,\phi},MF_{F,q_F}^{\rm mot,\phi},\Up(Q)$ supported on $U'$ and
$MF_{F,q_F}^{\rm mot,\phi},\Up(R)$ supported on $F$. Here the first
step of \eq{mo2eq26} combines \eq{mo2eq23}--\eq{mo2eq25}, the second
uses $\Up(Q)$ supported on $U'$ so that $\Up(Q)\od\Up(R)=\Up(Q)
\od\bigl(\Up(R)\vert_{U'}\bigr)=\Up(Q)\od\Up\bigl(R\vert_{U'}\bigr)$,
and the third uses $P\cong Q\ot_{Z_2}R\vert_{U'}$ since
$(E\vert_{U'},q)=(F,q_F)\op (L,q_L)$ and \eq{mo2eq19}. Equation
\eq{mo2eq26} is equivalent to the restriction of \eq{mo2eq22} to
$U'\subseteq U$. As we can cover $U$ by such Zariski open $U'$,
equation \eq{mo2eq22} holds. This proves the inductive step, and the
theorem.
\end{proof}

\begin{rem}{\bf(a)} the combination of equations \eq{mo2eq13} and
\eq{mo2eq22} in Theorems \ref{mo2thm3} and \ref{mo2thm4} will be
important in \S\ref{mo4}--\S\ref{mo5}. The main reason for passing
to the rings $\oM_X^{\hat\mu}$ is so that \eq{mo2eq22} holds.
\smallskip

\noindent{\bf(b)} Theorem \ref{mo2thm4} is more-or-less equivalent
to material in Kontsevich and Soibelman \cite[\S 5.1]{KoSo1}. It
implies that $MF_{E,q}^{\rm mot,\phi}$ depends only on $U,r,
\La^rE,\det(q)$, which is important in their definition of motivic
Donaldson--Thomas invariants. As for our $\oM_X^{\hat\mu}$,
Kontsevich and Soibelman \cite[\S 4.5]{KoSo1} also introduce an
extra relation in their ring of motives to make the analogue of
Theorem \ref{mo2thm4} true.
\smallskip

\noindent{\bf(c)} In equations \eq{mo2eq13}, \eq{mo2eq21} and
\eq{mo2eq22}, we can regard $\bL^{\dim U/2}\od MF_{E,q}^{\rm
mot,\phi}$ as the {\it relative motivic vanishing cycle\/} $MF_{E\ra
U,q}^{\rm mot,\phi,\rm rel}$ of $(E,q)$ relative to $U$.

In fact, as in Davison and Meinhardt \cite[\S 3.3]{DaMe}, one can
define such a relative vanishing cycle $MF_{E\ra U,q}^{\rm
mot,\phi,\rm rel}$ without assuming $U$ is smooth, only using that
$E\ra U$ is a smooth morphism. Then versions of Lemma \ref{mo2lem1}
and Theorem \ref{mo2thm4} hold without assuming $U$ is smooth. But
we will not need this.
\smallskip

\noindent{\bf(d)} We defined the ring $\oM_Y^{\hat\mu}$ by imposing
the pushforward $\phi_*$ of relation \eq{mo2eq19} in $\cM_X^{\hat\mu}$ under all morphisms $\phi:X\ra Y$. As in {\bf(c)}, we may rewrite  \eq{mo2eq22} as
\e
\Up(P)=MF_{E\ra U,q}^{\rm mot,\phi,\rm rel},
\label{mo2eq27}
\e
for $P\ra U$ the principal $\Z_2$-bundle corresponding to
$\bigl(\La^rE,\det(q)\bigr)$ under \eq{mo2eq20}, and where now both
sides of \eq{mo2eq27} make sense for $U$ singular as well as for $U$
smooth, by \cite[\S 3.3]{DaMe}. We may regard \eq{mo2eq27} as a
relation in $\cM_U^{\hat\mu}$, which is equivalent to \eq{mo2eq19}.
Thus, an alternative definition of the rings $\oM_Y^{\hat\mu}$,
closer in spirit to Kontsevich and Soibelman \cite[\S 4.5 \& \S
5.1]{KoSo1}, is to impose the relation in $\cM_Y^{\hat\mu}$ that for all $\K$-scheme morphisms $\phi:U\ra Y$, rank $r$ vector bundles $E\ra U$, and nondegenerate quadratic forms $q\in H^0(S^2E^*)$, the pushforward
$\phi_*\bigl(MF_{E\ra U,q}^{\rm mot,\phi,\rm rel}\bigr)$ depends only
on $U,\phi$ and~$\bigl(\La^rE,\det(q)\bigr)$.
\label{mo2rem4}
\end{rem}

\section{\texorpdfstring{Dependence of $MF^{\rm mot,\phi}_{U,f}$ on $f$}{Dependence of MFᵐᵒᵗᵠ on f}}
\label{mo3}

We will use the following notation, \cite[\S 4]{BBDJS}:

\begin{dfn} Let $U$ be a smooth $\K$-scheme, $f:U\ra\bA^1$ a
regular function, and $X=\Crit(f)$ as a closed $\K$-subscheme of
$U$. Write $I_X\subseteq \O_U$ for the sheaf of ideals of regular
functions $U\ra\bA^1$ vanishing on $X$, so that $I_X=I_{\d f}$. For
each $k=1,2,\ldots,$ write $X^{(k)}$ for the $k^{\rm th}$ {\it order
thickening of\/ $X$ in\/} $U$, that is, $X^{(k)}$ is the closed
$\K$-subscheme of $U$ defined by the vanishing of the sheaf of
ideals $I_X^k$ in $\O_U$. Also write $X^\red$ for the reduced
$\K$-subscheme of $U$, and $X^{(\iy)}$ or $\hat U$ for the formal
completion of $U$ along~$X$.

Then we have a chain of inclusions of closed $\K$-subschemes of $U$
\e
X^\red\subseteq X=X^{(1)}\subseteq X^{(2)}\subseteq X^{(3)}\subseteq
\cdots\subseteq X^{(\iy)}=\hat U\subseteq U,
\label{mo3eq1}
\e
although technically $X^{(\iy)}=\hat U$ is not a scheme, but a
formal scheme.

Write $f^{(k)}:=f\vert_{\smash{X^{(k)}}}:X^{(k)}\ra\K$, and
$f^\red:=f\vert_{\smash{X^\red}}:X^\red\ra\bA^1$, and $f^{(\iy)}$ or
$\hat f:=f\vert_{\smash{\hat U}}:\hat U\ra\bA^1$, so that
$f^{(k)},f^\red$ are regular functions on the $\K$-schemes
$X^{(k)},X^\red$, and $f^{(\iy)}=\hat f$ a formal function on the
formal $\K$-scheme $X^{(\iy)}=\hat U$. Note that
$f^\red:X^\red\ra\bA^1$ is locally constant, since~$X=\Crit(f)$.
\label{mo3def1}
\end{dfn}

As for perverse sheaves of vanishing cycles in Brav et al.\ \cite[\S
4]{BBDJS}, we can ask: how much of the sequence \eq{mo3eq1} does
$MF^{\rm mot,\phi}_{U,f}$ depend on? That is, is $MF^{\rm
mot,\phi}_{U,f}$ determined by $(X^\red,f^\red)$, or by
$(X^{(k)},f^{(k)})$ for some $k=1,2,\ldots,$ or by $(\hat U,\hat
f)$, as well as by $(U,f)$? Our next theorem shows that $MF^{\rm
mot,\phi}_{U,f}$ is determined by $(X^{(3)},f^{(3)})$, and hence a
fortiori also by $(X^{(k)},f^{(k)})$ for $k>3$ and by~$(\hat U,\hat
f)$:

\begin{thm} Let\/ $U,V$ be smooth\/ $\K$-schemes, $f:U\ra\bA^1,$
$g:V\ra\bA^1$ be regular functions, and\/ $X=\Crit(f),$ $Y=\Crit(g)$
as closed\/ $\K$-subschemes of\/ $U,V,$ so that the motivic
vanishing cycles $MF^{\rm mot,\phi}_{U,f},\,MF^{\rm mot,\phi}_{V,g}$
are defined on $X,Y$. Define $X^{(3)},f^{(3)}$ and\/
$Y^{(3)},g^{(3)}$ as in Definition\/ {\rm\ref{mo3def1},} and suppose
$\Phi:X^{(3)}\ra Y^{(3)}$ is an isomorphism with\/
$g^{(3)}\ci\Phi=f^{(3)},$ so that\/ $\Phi\vert_X:X\ra Y\subseteq
Y^{(3)}$ is an isomorphism. Then
\e
MF^{\rm mot,\phi}_{U,f}=\Phi\vert_X^*\bigl(MF^{\rm mot,\phi}_{V,g}
\bigr) \quad\text{in $\cM^{\hat\mu}_{X}$ and\/ $\oM^{\hat\mu}_{X}$.}
\label{mo3eq2}
\e

\label{mo3thm1}
\end{thm}

\begin{proof} The proof of Theorem \ref{mo3thm1} is based on the following result \cite[Prop. 4.3]{BBDJS}:

\begin{prop} Let\/ $U,V$ be smooth\/ $\K$-schemes,
$f:U\ra\bA^1,$ $g:V\ra\bA^1$ be regular functions, and\/
$X=\Crit(f)\subseteq U,$ $Y=\Crit(g)\subseteq V$. Using the notation
of Definition\/ {\rm\ref{mo3def1},} suppose $\Phi:X^{(k+1)}\ra
Y^{(k+1)}$ is an isomorphism with\/ $g^{(k+1)}\ci\Phi=f^{(k+1)}$ for
some $k\ge 2$. Then for each\/ $x\in X$ we can choose a smooth\/
$\K$-scheme $T$ and \'etale morphisms $\pi_U:T\ra U,$ $\pi_V:T\ra V$
such that\/
\begin{itemize}
\setlength{\itemsep}{0pt}
\setlength{\parsep}{0pt}
\item[{\bf(a)}] $e:=f\ci\pi_U=g\ci\pi_V:T\ra\bA^1;$
\item[{\bf(b)}] setting $Q=\Crit(e),$ then
$\pi_U\vert_{Q^{(k)}}:Q^{(k)}\ra X^{(k)}\subseteq U$ is an
isomorphism with a Zariski open neighbourhood\/ $\ti X^{(k)}$
of\/ $x$ in $X^{(k)};$ and
\item[{\bf(c)}] $\Phi\ci\pi_U\vert_{Q^{(k)}}=\pi_V
\vert_{Q^{(k)}}:Q^{(k)}\ra Y^{(k)}$.
\end{itemize}
\label{mo3prop1}
\end{prop}

Let $U,V,f,g,X,Y,\Phi$ be as in Theorem \ref{mo3thm1}, and $x\in X$.
Apply Proposition \ref{mo3prop1} with $k=2$ to get
$T,\pi_U,\pi_V,e,Q,\ti X^{(2)}$ satisfying (a)--(c). As
$\pi_U,\pi_V$ are \'etale with $e=f\ci\pi_U=g\ci\pi_V$, we see that
\begin{equation*}
\pi_U\vert_Q^*\bigl(MF_{U,f}^{\rm mot,\phi}\bigr)=MF_{T,e}^{\rm
mot,\phi}=\pi_V\vert_Q^*\bigl(MF_{V,g}^{\rm mot,\phi}\bigr).
\end{equation*}
Part (c) gives $\Phi\ci\pi_U\vert_{Q^{(2)}}= \pi_V \vert_{Q^{(2)}}$,
so restricting to $Q\subset Q^{(2)}$ yields $\Phi\vert_{\ti
X}\ci\pi_U\vert_Q= \pi_V\vert_Q$ for $\ti X=X\cap\ti X^{(2)}$. Thus
we see that
\begin{equation*}
\pi_U\vert_Q^*\bigl(MF_{U,f}^{\rm mot,\phi}\bigr)=\pi_U\vert_Q^*
\ci\Phi\vert_{\ti X}^*\bigl(MF_{V,g}^{\rm mot,\phi}\bigr).
\end{equation*}
But $\pi_U\vert_Q:Q\ra\ti X$ is an isomorphism by (b), so we may
omit $\pi_U\vert_Q^*$, yielding
\begin{equation*}
MF_{U,f}^{\rm mot,\phi}\big\vert_{\ti X}=
\Phi\vert_{\ti X}^*\bigl(MF_{V,g}^{\rm mot,\phi}\bigr).
\end{equation*}
This is the restriction of \eq{mo3eq2} to $\ti X\subseteq X$. Since
we can cover $X$ by such Zariski open $\ti X$, Theorem \ref{mo3thm1}
follows.
\end{proof}

\begin{rem} We can define motivic vanishing cycles
$MF_{\vphantom{U}\smash{{\hat U,\hat f}}}^{\rm mot,\phi}$ for a
class of formal functions $\hat f$ on formal schemes $\hat U$ using
Theorem \ref{mo3thm1}. Let $U$ be a smooth $\K$-scheme, $X\subseteq
U$ a closed $\K$-subscheme, and $\hat U$ the formal completion of
$U$ along $X$. Suppose $f:\hat U\ra\bA^1$ is a formal function with
$\Crit(f)=X\subseteq\hat U$. Then there is a unique
$MF_{\vphantom{U}\smash{{\hat U,\hat f}}}^{\rm mot,\phi}$ in
$\cM_X^{\hat\mu}$ or $\oM_X^{\hat\mu}$ with the property that if
$U'\subseteq U$ is Zariski open with $X'=X\cap U'$ and
$g:U'\ra\bA^1$ is regular with $g+I_{X'}^3=\hat f\vert_{\smash{\hat
U'}}+I_X^3$ in $H^0(\O_{U'}/I_{X'}^3)$ then
$MF_{\vphantom{U}\smash{{\hat U,\hat f}}}^{\rm
mot,\phi}\vert_{X'}=MF_{U',g}^{\rm mot,\phi}\vert_{X'}$. Theorem
\ref{mo3thm1} shows that $MF_{U',g}^{\rm mot,\phi}\vert_{X'}$ is
independent of the choice of $g$ with $g+I_{X'}^3=\hat
f\vert_{\smash{\hat U'}}+I_X^3$, so $MF_{\vphantom{U}\smash{{\hat
U,\hat f}}}^{\rm mot,\phi}$ is well-defined. Motivic Milnor fibres
for formal functions were also defined by Nicaise and Sebag
\cite{NiSe} in the context of formal geometry.
\label{mo3rem}
\end{rem}

In Brav et al.\ \cite[Th.~4.2]{BBDJS} we proved an analogue of
Theorem \ref{mo3thm1} for perverse sheaves of vanishing cycles. The
next example shows that Theorem \ref{mo3thm1} with
$X^{(2)},Y^{(2)},f^{(2)},g^{(2)}$ in place of
$X^{(3)},Y^{(3)},f^{(3)},g^{(3)}$ is false, so we cannot do better
than $(X^{(3)},f^{(3)})$ in Theorem \ref{mo3thm1}.

\begin{ex} Let $U=V=\bA^1\t(\bA^1\sm\{0\})$ as smooth
$\K$-schemes, define regular $f:U\ra\bA^1$ and $g:V\ra\bA^1$ by
$f(x,y)=x^2$ and $g(x,y)=x^2y.$ Then $X=\Crit(f)=\bigl\{(x,y)\in
V:x=0\bigr\}=\Crit(g)=Y$, as $y\ne 0$. Hence
$X^{(2)}=\bigl\{(x,y)\in V:x^2=0\bigr\}=Y^{(2)}$, so
$f^{(2)}=f\vert_{X^{(2)}}=0=g\vert_{Y^{(2)}}=g^{(2)}$. Thus
$\Phi=\id_{X^{(2)}}:X^{(2)}\ra Y^{(2)}$ is an isomorphism with
$g^{(2)}\ci\Phi=f^{(2)}$. However, using Lemma \ref{mo2lem1} it is
easy to show that
\begin{equation*}
MF_{U,f}^{\rm mot,\phi}\!=\!\bL^{-1/2}\!\od\!\Up(X\t\Z_2)\!\ne\!
\bL^{-1/2}\!\od\!\Up(P)\!=\!MF_{V,g}^{\rm mot,\phi}\!=\!\Phi\vert_X^*
\bigl(MF^{\rm mot,\phi}_{V,g}\bigr),
\end{equation*}
where $X\t\Z_2\ra X$ is the trivial principal $\Z_2$-bundle and
$P\ra Y$ the nontrivial principal $\Z_2$-bundle over
$X=Y=\bA^1\sm\{0\}$.
\label{mo3ex1}
\end{ex}

The next theorem is an immediate corollary of Theorem \ref{mo3thm1},
in which we take $U=V$, $X=Y$ and $\Phi=\id_{X^{(3)}}$.

\begin{thm} Let\/ $U$ be a smooth\/ $\K$-scheme and
$f,g:U\ra\bA^1$ regular functions. Suppose $X:=\Crit(f)=\Crit(g)$
and\/ $f^{(3)}=g^{(3)},$ that is, $f+I_X^3=g+I_X^3$ in
$H^0(\O_U/I_X^3),$ where $I_X\subseteq\O_U$ is the ideal of regular
functions vanishing on $X$. Then $MF_{U,f}^{\rm mot,\phi}=
MF_{U,g}^{\rm mot,\phi}$ in $\cM^{\hat\mu}_{X}$
and\/~$\oM^{\hat\mu}_{X}$.
\label{mo3thm2}
\end{thm}

\section{Stabilizing motivic vanishing cycles}
\label{mo4}

To set up notation for our main result, which is Theorem
\ref{mo4thm2} below, we need the following theorem, which is proved
in Joyce \cite[Prop.s 2.24 \& 2.25(c)]{Joyc2}.

\begin{thm}[Joyce \cite{Joyc2}] Let\/ $U,V$ be smooth\/
$\K$-schemes, $f:U\ra\bA^1,$ $g:V\ra\bA^1$ be regular, and\/
$X=\Crit(f),$ $Y=\Crit(g)$ as $\K$-subschemes of\/ $U,V$. Let\/
$\Phi:U\hookra V$ be a closed embedding of\/ $\K$-schemes with\/
$f=g\ci\Phi:U\ra\bA^1,$ and suppose $\Phi\vert_X:X\ra V\supseteq Y$
is an isomorphism $\Phi\vert_X:X\ra Y$. Then:
\smallskip

\noindent{\bf(i)} For each\/ $x\in X\subseteq U$ there exist a
Zariski open neighbourhood\/ $U'$ of\/ $x$ in $U,$ a smooth\/
$\K$-scheme $V',$ and morphisms $\jmath:V'\ra V,$ $\Phi':U'\ra V',$
$\al:V'\ra U',$ $\be:V'\ra\bA^n,$ and\/
$q_1,\ldots,q_n:U'\ra\bA^1\sm\{0\}$, where $n=\dim V-\dim U,$ such
that\/ $\jmath:V'\ra V$ and\/ $\al\t\be:V'\ra U'\t\bA^n$ are
\'etale, $\Phi\vert_{U'}=\jmath\ci\Phi',$ $\al\ci\Phi'=\id_{U'},$
$\be\ci\Phi'=0,$ and\/
\e
g\ci\jmath=f\ci\al+(q_1\ci\al)\cdot(z_1^2\ci\be)+\cdots+
(q_n\ci\al)\cdot(z_n^2\ci\be).
\label{mo4eq1}
\e
Thus, setting $f'=f\vert_{U'},$ $g'=g\ci\jmath,$ $X'=\Crit(f')=X\cap
U',$ and\/ $Y'=\Crit(g'),$ then $f'=g'\ci\Phi',$ and\/
$\Phi'\vert_{X'}:X'\ra Y',$ $\jmath\vert_{Y'}:Y'\ra Y,$
$\al\vert_{Y'}:Y'\ra X$ are \'etale. We also require
that\/~$\Phi\ci\al\vert_{Y'}=\jmath\vert_{Y'}:Y'\ra Y$.

\smallskip

\noindent{\bf(ii)} Write\/ $N_{\sst UV}$ for the normal bundle of\/
$\Phi(U)$ in $V,$ regarded as a vector bundle on $U$ in the exact
sequence
\e
\xymatrix@C=20pt{ 0 \ar[r] & TU \ar[rr]^(0.4){\d\Phi} && \Phi^*(TV)
\ar[rr]^(0.55){\Pi_{\sst UV}} && N_{\sst UV} \ar[r] & 0,}
\label{mo4eq2}
\e
so that $N_{\sst UV}\vert_X$ is a vector bundle on $X$. Then there
exists a unique $q_{\sst UV}\in H^0\bigl(S^2N_{\sst
UV}\vert_X^*\bigr)$ which is a nondegenerate quadratic form on
$N_{\sst UV}\vert_X,$ such that whenever
$x,U',V',\jmath,\Phi',\al,\be,n,q_a$ are as in {\bf(i)\rm,} then
there is an isomorphism $\hat\be:\langle\d z_1,\ldots,\d
z_n\rangle_{U'}\ab\ra N_{\sst UV}^*\vert_{U'}$ making the following
commute:
\begin{equation*}
\xymatrix@C=110pt@R=15pt{
*+[r]{N_{\sst UV}^*\vert_{U'}} \ar[r]_(0.3){\Pi_{\sst UV}^*\vert_{U'}}
& *+[l]{\Phi\vert_{U'}^*(T^*V)=\Phi^{\prime *}\ci
\jmath^*(T^*V)} \ar[d]_{\Phi^{\prime *}(\d\jmath^*)} \\
*+[r]{\langle\d z_1,\ldots,\d z_n\rangle_{U'}=\Phi^{\prime *}
\ci\be^*(T_0^*\bA^n)} \ar[r]^(0.7){\Phi^{\prime *}(\d\be^*)}
\ar@{.>}[u]_{\hat\be} & *+[l]{\Phi^{\prime *}(T^*V'),}}
\end{equation*}
and if\/ $X'=X\cap U',$ then
\begin{equation*}
q_{\sst UV}\vert_{X'}=\bigl[q_1\cdot S^2\hat\be(\d z_1\ot\d z_1)+
\cdots+q_n\cdot S^2\hat\be(\d z_n\ot\d z_n)\bigr]\big\vert_{X'}.
\end{equation*}
\label{mo4thm1}
\end{thm}

Following \cite[Def.s 2.26 \& 2.34]{Joyc2}, we define:

\begin{dfn} Let $U,V$ be smooth $\K$-schemes, $f:U\ra\bA^1$,
$g:V\ra\bA^1$ be regular, and $X=\Crit(f)$, $Y=\Crit(g)$ as
$\K$-subschemes of $U,V$. Suppose $\Phi:U\hookra V$ is an embedding
of $\K$-schemes with $f=g\ci\Phi:U\ra\bA^1$ and $\Phi\vert_X:X\ra Y$
an isomorphism. Then Theorem \ref{mo4thm1}(ii) defines the normal
bundle $N_{\sst UV}$ of $U$ in $V$, a vector bundle on $U$ of rank
$n=\dim V-\dim U$, and a nondegenerate quadratic form $q_{\sst
UV}\in H^0(S^2N_{\sst UV}^*\vert_X)$. Taking top exterior powers in
the dual of \eq{mo4eq2} gives an isomorphism of line bundles on $U$
\begin{equation*}
\rho_{\sst UV}:K_U\ot\La^nN_{\sst UV}^*
\,{\buildrel\cong\over\longra}\,\Phi^*(K_V),
\end{equation*}
where $K_U,K_V$ are the canonical bundles of $U,V$.

Write $X^\red$ for the reduced $\K$-subscheme of $X$. As $q_{\sst
UV}$ is a nondegenerate quadratic form on $N_{\sst UV}\vert_X,$ its
determinant $\det(q_{\sst UV})$ is a nonzero section of
$(\La^nN_{\sst UV}^*)\vert_X^{\ot^2}$. Define an isomorphism of line
bundles on~$X^\red$:
\e
J_\Phi=\rho_{\sst UV}^{\ot^2}\ci
\bigl(\id_{K_U^2\vert_{X^\red}}\ot\det(q_{\sst
UV})\vert_{X^\red}\bigr):K_U^{\ot^2}\big\vert_{X^\red}
\,{\buildrel\cong\over\longra}\,\Phi\vert_{X^\red}^*
\bigl(K_V^{\ot^2}\bigr).
\label{mo4eq3}
\e

Since principal $\Z_2$-bundles $\pi:P\ra X$ are a topological
notion, and $X^\red$ and $X$ have the same topological space,
principal $\Z_2$-bundles on $X^\red$ and on $X$ are equivalent.
Define $\pi_\Phi:P_\Phi\ra X$ to be the principal $\Z_2$-bundle
which parametrizes square roots of $J_\Phi$ on $X^\red$. That is,
local sections $s_\al:X\ra P_\Phi$ of $P_\Phi$ correspond to local
isomorphisms $\al: K_U\vert_{X^\red} \ra\Phi\vert_{X^\red}^*(K_V)$
on $X^\red$ with $\al\ot\al=J_\Phi$. Equivalently, $P_\Phi$ is the
principal $\Z_2$-bundle corresponding to $\bigl(\La^nN_{\sst
UV}\vert_X, \det(q_{\sst UV})\bigr)$ under the 1-1
correspondence~\eq{mo2eq20}.
\label{mo4def1}
\end{dfn}

The reason for restricting to $X^\red$ above is the next result
\cite[Prop.~2.27]{Joyc2}, whose proof uses the fact that $X^\red$ is
reduced in an essential way.

\begin{lem} In Definition\/ {\rm\ref{mo4def1},} the isomorphism
$J_\Phi$ in \eq{mo4eq3} and the principal\/ $\Z_2$-bundle\/
$\pi_\Phi:P_\Phi\ra X$ depend only on $U,\ab V,\ab X,\ab Y,\ab f,g$
and\/ $\Phi\vert_X:X\ra Y$. That is, they do not depend on
$\Phi:U\ra V$ apart from\/~$\Phi\vert_X:X\ra Y$.
\label{mo4lem1}
\end{lem}

Using the notation of Definition \ref{mo4def1}, we can state our
main result:

\begin{thm} Let\/ $U,V$ be smooth\/ $\K$-schemes, $f:U\ra\bA^1,$
$g:V\ra\bA^1$ be regular, and\/ $X=\Crit(f),$ $Y=\Crit(g)$ as
$\K$-subschemes of\/ $U,V$. Let\/ $\Phi:U\hookra V$ be an embedding
of\/ $\K$-schemes with\/ $f=g\ci\Phi:U\ra\bA^1,$ and suppose
$\Phi\vert_X:X\ra V\supseteq Y$ is an isomorphism $\Phi\vert_X:X\ra
Y$. Then
\e
\Phi\vert_X^*\bigl(MF^{\rm mot, \phi}_{V,g}\bigr)=MF^{\rm mot,
\phi}_{U,f}\od \Up (P_\Phi)\quad\text{in\/ $\oM^{\hat\mu}_X,$}
\label{mo4eq4}
\e
where\/ $P_\Phi$ is as in Definition\/ {\rm\ref{mo4def1},} and\/
$\Up$ is defined in\/~\eq{mo2eq18}.
\label{mo4thm2}
\end{thm}

In Brav et al.\ \cite[Th.~5.4]{BBDJS} we proved an analogue of
Theorem \ref{mo4thm2} for perverse sheaves of vanishing cycles.

\subsection{Proof of Theorem \ref{mo4thm2}}
\label{mo41}

Suppose $U,V,f,g,X,Y,\Phi$ are as in Theorem \ref{mo4thm2}, and use
the notation $N_{\sst UV},q_{\sst UV}$ from Theorem
\ref{mo4thm1}(ii) and $J_\Phi,P_\Phi$ from Definition \ref{mo4def1}.
Let $x,U',V',\jmath,\Phi',\ab\al,\ab\be,\ab q_1,\ab\ldots,\ab
q_n,f',g',X',Y'$ be as in Theorem \ref{mo4thm1}(i). Then in
$\oM^{\hat\mu}_{X'}$ we have
\ea
\Phi\vert_X^*&\bigl(MF^{\rm mot, \phi}_{V,g}\bigr)\big\vert_{X'}=
\Phi'\vert_X^*\ci\jmath\vert_{Y'}^*\bigl(MF^{\rm mot, \phi}_{V,g}\bigr)
=\Phi'\vert_X^*\bigl(MF^{\rm mot,\phi}_{V',g'}\bigr)
\nonumber\\
&=\Phi'\vert_X^*\ci(\al\t\be)^*\bigl(MF^{\rm mot,\phi}_{U'\t\bA^n,f\ci
\pi_{U'}+\Si_{i=1}^n(q_i\ci\pi_{U'})\cdot (z_i^2\ci \pi_{\bA^n})}\bigr)
\nonumber\\
&=(\id_{X'}\t 0)^*\bigl(MF^{\rm mot,\phi}_{U',f'}\od
\bL^{\dim U/2}\od MF^{\rm mot,\phi}_{U'\t\bA^n,
\Si_{i=1}^n(q_i\ci\pi_{U'})\cdot (z_i^2\ci \pi_{\bA^n})}\bigr)
\nonumber\\
&=(\id_{X'}\t 0)^*\bigl(MF^{\rm mot,\phi}_{U',f'}\od \Up(P_{q_1\cdots q_n})
\bigr)=MF^{\rm mot,\phi}_{U',f'}\od \Up\bigl(P_{q_1\cdots q_n}
\vert_{X'}\bigr)
\nonumber\\
&=MF^{\rm mot,\phi}_{U',f'}\od \Up\bigl(P_\Phi\vert_{X'}\bigr)=
\bigl(MF^{\rm mot,\phi}_{U,f}\od \Up (P_\Phi)\bigr)\big\vert_{X'},
\label{mo4eq5}
\ea
using $\Phi\vert_{U'}=\jmath\ci\Phi'$ in the first step,
$\jmath:V'\ra V$ \'etale with $g'=g\ci\jmath$ in the second,
$\al\t\be:V'\ra U'\t\bA^n$ \'etale and \eq{mo4eq1} in the third, and
$\al\ci\Phi'=\id_{U'}$, $\be\ci\Phi'=0$, and Theorem \ref{mo2thm3}
in the fourth.

In the fifth step of \eq{mo4eq5}, we apply Theorem \ref{mo2thm4} to
the vector bundle $U'\t\bA^n\ra U'$ and nondegenerate quadratic form
$\sum_{i=1}^n(q_i\ci\pi_{U'})\cdot (z_i^2\ci \pi_{\bA^n})$, and we
write $P_{q_1\cdots q_n}\ra U'$ for the principal $\Z_2$-bundle
corresponding to $\bigl(\O_{U'},q_1\cdots q_n\bigr)$ under
\eq{mo2eq20}. The sixth step uses $MF^{\rm mot,\phi}_{U',f'}$
supported on $X'\cong X'\t\{0\}$, the seventh that $P_{q_1\cdots
q_n} \vert_{X'}\cong P_\Phi\vert_{X'}$ since Theorem
\ref{mo4thm1}(ii) implies an identification between $q_1\cdots q_n$
and $\det(q_{\sst UV})$ on $X'$ and $P_{q_1\cdots q_n} \vert_{X'},
P_\Phi\vert_{X'}$ parametrize square roots of $q_1\cdots q_n$ and
$\det(q_{\sst UV})$ on $X'$, and the eighth that $U'\subseteq U$ is
open with $f'=f\vert_{U'}$. Equation \eq{mo4eq5} proves the
restriction of \eq{mo4eq4} to the Zariski open set $X'\subseteq X$.
Since we can cover $X$ by such open $X'$, Theorem \ref{mo4thm2}
follows.

\section{Motivic vanishing cycles on d-critical loci}
\label{mo5}

In \S\ref{mo51} we introduce {\it d-critical loci\/} from Joyce
\cite{Joyc2}. Our main result Theorem \ref{mo5thm5}, which
associates a motive $MF_{X,s}\in\oM^{\hat\mu}_X$ to each oriented
d-critical locus $(X,s)$, is stated and discussed in \S\ref{mo52},
and proved in \S\ref{mo54}. Section \ref{mo53} describes a torus
localization formula for $MF_{X,s}$ due to Maulik~\cite{Maul}.

In \S\ref{mo51} we only need our $\K$-schemes to be locally of
finite type, not of finite type. In \S\ref{mo52}--\S\ref{mo54} we
take $\K$-schemes to be finite type, but discuss extension to the
locally of finite type case in Remark~\ref{mo5rem}.

\subsection{Background material on d-critical loci}
\label{mo51}

Here are the main definitions and results on d-critical loci, from
Joyce \cite[Th.s 2.1, 2.20, 2.28 \& Def.s 2.5, 2.18, 2.31 \& Ex.\
2.38]{Joyc2}. In fact \cite{Joyc2} develops two versions of the
theory, {\it algebraic d-critical loci\/} on $\K$-schemes and {\it
complex analytic d-critical loci\/} on complex analytic spaces, but
we discuss only the former.

\begin{thm} Let\/ $X$ be a $\K$-scheme, which must be locally of
finite type, but not necessarily of finite type. Then there exists a
sheaf\/ $\cS_X$ of\/ $\K$-vector spaces on $X,$ unique up to
canonical isomorphism, which is uniquely characterized by the
following two properties:
\begin{itemize}
\setlength{\itemsep}{0pt}
\setlength{\parsep}{0pt}
\item[{\bf(i)}] Suppose $R\subseteq X$ is Zariski open, $U$ is a
smooth\/ $\K$-scheme, and\/ $i:R\hookra U$ is a closed
embedding. Then we have an exact sequence of sheaves of\/
$\K$-vector spaces on $R\!:$
\begin{equation*}
\smash{\xymatrix@C=30pt{ 0 \ar[r] & I_{R,U} \ar[r] &
i^{-1}(\O_U) \ar[r]^{i^\sharp} & \O_X\vert_R \ar[r] & 0, }}
\end{equation*}
where $\O_X,\O_U$ are the sheaves of regular functions on $X,U,$
and\/ $i^\sharp$ is the morphism of sheaves of\/ $\K$-algebras
on $R$ induced by $i$.

There is an exact sequence of sheaves of\/ $\K$-vector spaces on
$R\!:$
\begin{equation*}
\xymatrix@C=20pt{ 0 \ar[r] & \cS_X\vert_R
\ar[rr]^(0.4){\io_{R,U}} &&
\displaystyle\frac{i^{-1}(\O_U)}{I_{R,U}^2} \ar[rr]^(0.4)\d &&
\displaystyle\frac{i^{-1}(T^*U)}{I_{R,U}\cdot i^{-1}(T^*U)}\,, }
\end{equation*}
where $\d$ maps $f+I_{R,U}^2\mapsto \d f+I_{R,U}\cdot
i^{-1}(T^*U)$.
\item[{\bf(ii)}] Let\/ $R\subseteq S\subseteq X$ be Zariski open,
$U,V$ be smooth\/ $\K$-schemes, $i:R\hookra U,$ $j:S\hookra V$
closed embeddings, and\/ $\Phi:U\ra V$ a morphism with\/
$\Phi\ci i=j\vert_R:R\ra V$. Then the following diagram of
sheaves on $R$ commutes:
\e
\begin{gathered}
\xymatrix@C=12pt{ 0 \ar[r] & \cS_X\vert_R \ar[d]^\id
\ar[rrr]^(0.4){\io_{S,V}\vert_R} &&&
\displaystyle\frac{j^{-1}(\O_V)}{I_{S,V}^2}\Big\vert_R
\ar@<-2ex>[d]^{i^{-1}(\Phi^\sharp)} \ar[rr]^(0.4)\d &&
\displaystyle\frac{j^{-1}(T^*V)}{I_{S,V}\cdot
j^{-1}(T^*V)}\Big\vert_R \ar@<-2ex>[d]^{i^{-1}(\d\Phi)} \\
 0 \ar[r] & \cS_X\vert_R \ar[rrr]^(0.4){\io_{R,U}} &&&
\displaystyle\frac{i^{-1}(\O_U)}{I_{R,U}^2} \ar[rr]^(0.4)\d &&
\displaystyle\frac{i^{-1}(T^*U)}{I_{R,U}\cdot i^{-1}(T^*U)}\,.
}\!\!\!\!\!\!\!{}
\end{gathered}
\label{mo5eq1}
\e
Here $\Phi:U\ra V$ induces $\Phi^\sharp: \Phi^{-1}(\O_V)\ra\O_U$
on $U,$ so we have
\e
i^{-1}(\Phi^\sharp):j^{-1}(\O_V)\vert_R=i^{-1}\ci
\Phi^{-1}(\O_V)\longra i^{-1}(\O_U),
\label{mo5eq2}
\e
a morphism of sheaves of\/ $\K$-algebras on $R$. As $\Phi\ci
i=j\vert_R,$ equation \eq{mo5eq2} maps $I_{S,V}\vert_R\ra
I_{R,U},$ and so maps $I_{S,V}^2\vert_R\ra I_{R,U}^2$. Thus
\eq{mo5eq2} induces the morphism in the second column of\/
\eq{mo5eq1}. Similarly, $\d\Phi:\Phi^{-1}(T^*V)\ra T^*U$ induces
the third column of\/~\eq{mo5eq1}.
\end{itemize}

There is a natural decomposition\/ $\cS_X=\cSz_X\op\K_X,$ where\/
$\K_X$ is the constant sheaf on $X$ with fibre $\K,$ and\/
$\cSz_X\subset\cS_X$ is the kernel of the composition
\begin{equation*}
\xymatrix@C=40pt{ \cS_X \ar[r] & \O_X
\ar[r]^(0.47){i_X^\sharp} & \O_{X^\red}, }
\end{equation*}
with\/ $X^\red$ the reduced\/ $\K$-subscheme of\/ $X,$ and\/
$i_X:X^\red\hookra X$ the inclusion.
\label{mo5thm1}
\end{thm}

\begin{dfn} An {\it algebraic d-critical locus\/} over a
field $\K$ is a pair $(X,s)$, where $X$ is a $\K$-scheme, locally of
finite type, but not necessarily of finite type, and $s\in
H^0(\cSz_X)$ for $\cSz_X$ as in Theorem \ref{mo5thm1}, such that for
each $x\in X$, there exists a Zariski open neighbourhood $R$ of $x$
in $X$, a smooth $\K$-scheme $U$, a regular function
$f:U\ra\bA^1=\K$, and a closed embedding $i:R\hookra U$, such that
$i(R)=\Crit(f)$ as $\K$-subschemes of $U$, and
$\io_{R,U}(s\vert_R)=i^{-1}(f)+I_{R,U}^2$. We call the quadruple
$(R,U,f,i)$ a {\it critical chart\/} on~$(X,s)$.

Let $(X,s)$ be an algebraic d-critical locus, and $(R,U,f,i)$ a
critical chart on $(X,s)$. Let $U'\subseteq U$ be Zariski open, and
set $R'=i^{-1}(U')\subseteq R$, $i'=i\vert_{R'}:R'\hookra U'$, and
$f'=f\vert_{U'}$. Then $(R',U',f',i')$ is a critical chart on
$(X,s)$, and we call it a {\it subchart\/} of $(R,U,f,i)$. As a
shorthand we write~$(R',U',f',i')\subseteq (R,U,f,i)$.

Let $(R,U,f,i),(S,V,g,j)$ be critical charts on $(X,s)$, with
$R\subseteq S\subseteq X$. An {\it embedding\/} of $(R,U,f,i)$ in
$(S,V,g,j)$ is a locally closed embedding $\Phi:U\hookra V$ such
that $\Phi\ci i=j\vert_R$ and $f=g\ci\Phi$. As a shorthand we write
$\Phi: (R,U,f,i)\hookra(S,V,g,j)$. If $\Phi:(R,U,f,i)\hookra
(S,V,g,j)$ and $\Psi:(S,V,g,j)\hookra(T,W,h,k)$ are embeddings, then
$\Psi\ci\Phi:(R,U,i,e)\hookra(T,W,h,k)$ is also an embedding.
\label{mo5def1}
\end{dfn}

\begin{thm} Let\/ $(X,s)$ be an algebraic d-critical locus, and\/
$(R,U,f,i),\ab(S,\ab V,\ab g,\ab j)$ be critical charts on $(X,s)$.
Then for each\/ $x\in R\cap S\subseteq X$ there exist subcharts
$(R',U',f',i')\subseteq(R,U,f,i),$ $(S',V',g',j')\subseteq
(S,V,g,j)$ with\/ $x\in R'\cap S'\subseteq X,$ a critical chart\/
$(T,W,h,k)$ on $(X,s),$ and embeddings $\Phi:(R',U',f',i')\hookra
(T,W,h,k),$ $\Psi:(S',V',g',j')\hookra(T,W,h,k)$.
\label{mo5thm2}
\end{thm}

\begin{thm} Let\/ $(X,s)$ be an algebraic d-critical locus, and\/
$X^\red\subseteq X$ the associated reduced\/ $\K$-scheme. Then there
exists a line bundle $K_{X,s}$ on $X^\red$ which we call the
\begin{bfseries}canonical bundle\end{bfseries} of\/ $(X,s),$ which
is natural up to canonical isomorphism, and is characterized by the
following properties:
\begin{itemize}
\setlength{\itemsep}{0pt}
\setlength{\parsep}{0pt}
\item[{\bf(i)}] If\/ $(R,U,f,i)$ is a critical chart on
$(X,s),$ there is a natural isomorphism
\e
\io_{R,U,f,i}:K_{X,s}\vert_{R^\red}\longra
i^*\bigl(K_U^{\ot^2}\bigr)\vert_{R^\red},
\label{mo5eq3}
\e
where $K_U=\La^{\dim U}T^*U$ is the canonical bundle of\/ $U$ in
the usual sense.
\item[{\bf(ii)}] Let\/ $\Phi:(R,U,f,i)\hookra(S,V,g,j)$ be an
embedding of critical charts on $(X,s)$. Then \eq{mo4eq3}
defines an isomorphism of line bundles on $\Crit(f)^\red:$
\begin{equation*}
J_\Phi:K_U^{\ot^2}\vert_{\Crit(f)^\red}\,{\buildrel\cong\over\longra}\,
\Phi\vert_{\Crit(f)^\red}^*\bigl(K_V^{\ot^2}\bigr).
\end{equation*}
Since $i:R\ra\Crit(f)$ is an isomorphism with\/ $\Phi\ci
i=j\vert_R,$ this gives
\begin{equation*}
i\vert_{R^\red}^*(J_\Phi):i^*\bigl(K_U^{\ot^2}\bigr)\big\vert_{R^\red}
\,{\buildrel\cong\over\longra}\,j^*\bigl(K_V^{\ot^2}\bigr)\big\vert_{R^\red},
\end{equation*}
and we must have
\e
\io_{S,V,g,j}\vert_{R^\red}=i\vert_{R^\red}^*(J_\Phi)\ci
\io_{R,U,f,i}:K_{X,s}\vert_{R^\red}\longra j^*
\bigl(K_V^{\ot^2}\bigr)\big\vert_{R^\red}.
\label{mo5eq4}
\e
\end{itemize}
\label{mo5thm3}
\end{thm}

\begin{dfn} Let $(X,s)$ be an algebraic d-critical locus, and
$K_{X,s}$ its canonical bundle from Theorem \ref{mo5thm3}. An {\it
orientation\/} on $(X,s)$ is a choice of square root line bundle
$K_{X,s}^{1/2}$ for $K_{X,s}$ on $X^\red$. That is, an orientation
is a line bundle $L$ on $X^\red$, together with an isomorphism
$L^{\ot^2}=L\ot L\cong K_{X,s}$. A d-critical locus with an
orientation will be called an {\it oriented d-critical locus}.
\label{mo5def2}
\end{dfn}

\begin{ex} Let $X$ be a smooth $\K$-scheme. Then in Theorem
\ref{mo5thm1} we have $\cS_X\cong\K_X$ and $\cSz_X\cong 0$. The
section $s=0\in H^0(\cSz_X)$ makes $(X,0)$ into an algebraic
d-critical locus, covered by the critical chart
$(R,U,f,i)=(X,X,0,\id_X)$. Theorem \ref{mo5thm3}(i) for this chart
shows that $K_{X,0}\cong K_X^{\ot^2}$, where $K_X$ is the usual
canonical bundle of $X$. Thus, $(X,0)$ has a natural
orientation~$K_{X,0}^{1/2}=K_X$.

As we call $K_{X,0}$ the canonical bundle of $(X,0)$, one might have
expected $K_{X,0}\cong K_X$. The explanation is that as a derived
scheme, $\Crit(0:X\ra\bA^1)$ is not $X$, but the shifted cotangent
bundle $T^*X[1]$, and the degree $-1$ fibres of the projection
$T^*X[1]\ra X$ include an extra factor of $K_X$ in~$K_{X,0}$.
\label{mo5ex}
\end{ex}

In \cite[Th.~6.6]{BBJ} we show that algebraic d-critical loci are
classical truncations of objects in derived algebraic geometry known
as $-1$-{\it shifted symplectic derived schemes}, introduced by
Pantev, To\"en, Vaqui\'e and Vezzosi \cite{PTVV}.

\begin{thm}[Bussi, Brav and Joyce \cite{BBJ}] Suppose\/ $(\bs
X,\om)$ is a $-1$-shifted symplectic derived scheme over a field\/
$\K$ in the sense of Pantev et al.\ {\rm\cite{PTVV},} and let\/
$X=t_0(\bs X)$ be the associated classical\/ $\K$-scheme of\/ ${\bs
X}$. Then $X$ extends naturally to an algebraic d-critical locus\/
$(X,s)$. The canonical bundle $K_{X,s}$ from Theorem\/
{\rm\ref{mo5thm3}} is naturally isomorphic to the determinant line
bundle $\det(\bL_{\bs X})\vert_{X^\red}$ of the cotangent complex\/
$\bL_{\bs X}$ of\/~$\bs X$.
\label{mo5thm4}
\end{thm}

Pantev et al.\ \cite{PTVV} show that derived moduli schemes of
coherent sheaves on a Calabi--Yau 3-fold have $-1$-shifted
symplectic structures, giving \cite[Cor.~6.7]{BBJ}:

\begin{cor} Suppose $Y$ is a Calabi--Yau\/ $3$-fold over\/ $\K,$
and\/ $\cM$ is a classical moduli\/ $\K$-scheme of simple coherent
sheaves in $\coh(Y),$ or simple complexes of coherent sheaves in
$D^b\coh(Y),$ with perfect obstruction theory\/
$\phi:\cE^\bu\ra\bL_\cM$ as in Thomas\/ {\rm\cite{Thom}} or
Huybrechts and Thomas\/ {\rm\cite{HuTh}}. Then $\cM$ extends
naturally to an algebraic d-critical locus $(\cM,s)$. The canonical
bundle $K_{\cM,s}$ from Theorem\/ {\rm\ref{mo5thm3}} is naturally
isomorphic to $\det(\cE^\bu)\vert_{\cM^\red}$.
\label{mo5cor1}
\end{cor}

Here we call $F\in\coh(Y)$ {\it simple\/} if $\Hom(F,F)=\K$, and
$F^\bu$ in $D^b\coh(Y)$ {\it simple\/} if $\Hom(F^\bu,F^\bu)=\K$ and
$\mathop{\rm Ext}^{<0}(F^\bu,F^\bu)=0$. Thus, d-critical loci will
have applications in Donaldson--Thomas theory for Calabi--Yau
3-folds \cite{JoSo,KoSo1,KoSo2,Thom}. Orientations on $(\cM,s)$ are
closely related to {\it orientation data\/} in the work of
Kontsevich and Soibelman~\cite{KoSo1,KoSo2}.

Pantev et al.\ \cite{PTVV} also show that derived intersections
$L\cap M$ of (derived) algebraic Lagrangians $L,M$ in an algebraic
symplectic manifold $(S,\om)$ have $-1$-shifted symplectic
structures, so that Theorem \ref{mo5thm4} gives them the structure
of algebraic d-critical loci. Thus we may
deduce~\cite[Cor.~6.8]{BBJ}:

\begin{cor} Suppose $(S,\om)$ is an algebraic symplectic manifold
over $\K,$ and\/ $L,M$ are algebraic Lagrangians in $S$. Then the
intersection $X=L\cap M,$ as a $\K$-subscheme of\/ $S,$ extends
naturally to an algebraic d-critical locus\/ $(X,s)$. The canonical
bundle $K_{X,s}$ from Theorem\/ {\rm\ref{mo5thm3}} is isomorphic
to\/ $K_L\vert_{X^\red}\ot K_M\vert_{X^\red}$.

\label{mo5cor2}
\end{cor}

\subsection{The main result, and applications}
\label{mo52}

Here is our main result, which will be proved in~\S\ref{mo54}.

\begin{thm} Let\/ $(X,s)$ be a finite type algebraic d-critical
locus with a choice of orientation $K_{X,s}^{1/2}$. There exists a
unique motive $MF_{X,s}\in\oM^{\hat\mu}_X$ with the property that
if\/ $(R,U,f,i)$ is a critical chart on $(X,s),$ then
\e
MF_{X,s}\vert_R=i^*\bigl(MF_{U,f}^{\rm mot,\phi}\bigr)\od\Up
(Q_{R,U,f,i})\quad\text{in\/ $\oM^{\hat\mu}_R,$}
\label{mo5eq5}
\e
where $Q_{R,U,f,i}\ra R$ is the principal\/ $\Z_2$-bundle
parametrizing local isomorphisms $\al:K_{X,s}^{1/2}
\vert_{R^\red}\ra i^*(K_U)\vert_{R^\red}$ with\/ $\al\ot\al=
\io_{R,U,f,i},$ for $\io_{R,U,f,i}$ as in\/ \eq{mo5eq3}.
\label{mo5thm5}
\end{thm}

\begin{rem} The theory of algebraic d-critical loci $(X,s)$ in
\cite{Joyc2} is developed for $\K$-schemes $X$ locally of finite
type, but not necessarily of finite type. The construction of
$MF_{X,s}$ in Theorem \ref{mo5thm5} is local in $X$. Thus, we can
easily extend Theorem \ref{mo5thm5} to $X$ only locally of finite
type, provided we have suitable generalizations of the motivic rings
$K_0(\Sch_X),\ab\cM_X,\ab K_0^{\hat\mu}(\Sch_X),\ab
\cM^{\hat\mu}_X,\oM^{\hat\mu}_X$ of \S\ref{mo2} to $\K$-schemes $X$
only locally of finite type.

Here are two natural ways to extend the groups $K_0(\Sch_X)$ to $X$ locally of finite type, based on the `stack functions' $\mathop{\rm SF}(X)$ and
`local stack functions' $\mathop{\rm SF}(X)$ of Joyce \cite[Def.s 3.1 \& 3.9]{Joyc1}. We can also generalize $\cM_X,\ab K_0^{\hat\mu}(\Sch_X),\ab
\cM^{\hat\mu}_X$ and $\oM^{\hat\mu}_X$ in the same ways.
\begin{itemize}
\setlength{\itemsep}{0pt}
\setlength{\parsep}{0pt}
\item[(i)] We could define $K_0(\Sch_X)$ to be generated by symbols $[S]$ for $S$ a finite type $\K$-scheme and $\Pi_S^X:S\ra X$ a morphism, with relation $[S]=[T]+[S\sm T]$ if $T\subseteq S$ is a closed $\K$-subscheme.
Note that $S$ must be finite type as a $\K$-scheme, {\it not\/} as an $X$-scheme.

If $X$ is not of finite type then $K_0(\Sch_X)$ is a ring without identity, as $[X]$ is not an element of $K_0(\Sch_X)$.

If $X,Y$ are locally of finite type and $\phi:X\ra Y$ a morphism, then pushforwards $\phi_*:K_0(\Sch_X)\ra K_0(\Sch_Y)$ are defined for arbitrary $\phi$, but pullbacks $\phi^*:K_0(\Sch_Y)\ra K_0(\Sch_X)$ are
defined only for $\phi$ of finite type.
\item[(ii)] We could define elements of $K_0(\Sch_X)$ to be $\sim$-equivalence classes of sums $\sum_{i\in I}c_i[S_i]$, where $I$ is a possibly infinite indexing set, $\Pi_{S_i}^X:S_i\ra X$ is a finite type morphism of $\K$-schemes for each $i\in I$ (so that $S_i$ is locally of finite type as a $\K$-scheme), and $c_i\in\Z$ for $i\in I$, such that for any finite type $\K$-subscheme $X'\subseteq X$, we have $S_i\t_XX'\ne\es$ for only finitely many $i\in I$. Define an equivalence relation $\sim$ on such sums by $\sum_{i\in I}c_i[S_i]\sim\sum_{j\in J}d_j[T_j]$ if for all finite type $\K$-subschemes $X'\subseteq X$, we have $\sum_{i\in I}c_i[S_i\t_XX']=\sum_{j\in J}d_j[T_j\t_XX']$ in $K_0(\Sch_{X'})$, where $K_0(\Sch_{X'})$ is as in \S\ref{mo21} as $X'$ is of finite type, and the sums make sense as there are only finitely many nonzero terms.

Then $K_0(\Sch_X)$ is a ring with identity $[X]$, since $\id_X:X\ra X$ is a finite type morphism. Pullbacks $\phi^*:K_0(\Sch_Y)\ra K_0(\Sch_X)$ are defined for arbitrary $\phi:X\ra Y$, but pushforwards $\phi_*:K_0(\Sch_X)\ra K_0(\Sch_Y)$ are defined only for $\phi$ of finite type.
\end{itemize}

For the purposes of generalizing Theorem \ref{mo5thm5}, (i) does not work -- for $X$ only locally of finite type, $MF_{X,s}$ would in general not be supported on a finite type subscheme of $X$, and so could not be an element of $\oM^{\hat\mu}_X$ defined as in (i). But (ii) does work, and defining $\oM^{\hat\mu}_X$ using the method of (ii), Theorem \ref{mo5thm5} extends to $X$ locally of finite type in a straightforward way. Note that we cannot push $MF_{X,s}$ forward to $\oM_\K$ if $X$ is not of finite type, since $\pi:X\ra\Spec\K$ is not a finite type morphism, and $\pi_*:\oM_X\ra\oM_\K$ is not defined.
\label{mo5rem}
\end{rem}

In Brav et al.\ \cite[Th.~6.9]{BBDJS} we proved an analogue of
Theorem \ref{mo5thm5} for perverse sheaves of vanishing cycles.
Theorem \ref{mo5thm4} and Corollaries \ref{mo5cor1} and
\ref{mo5cor2} imply:

\begin{cor} Let\/ $(\bs X,\om)$ be a $-1$-shifted symplectic
derived scheme over\/ $\K$ in the sense of Pantev et al.\
{\rm\cite{PTVV},} and\/ $X=t_0(\bs X)$ the associated classical\/
$\K$-scheme, assumed of finite type. Suppose we are given a square
root\/ $\smash{\det(\bL_{\bs X})\vert_X^{1/2}}$ for $\det(\bL_{\bs
X})\vert_X$. Then we may define a natural motive $MF_{\bs X,\om}\in
\oM^{\hat\mu}_X$.
\label{mo5cor3}
\end{cor}

\begin{cor} Suppose $Y$ is a Calabi--Yau\/ $3$-fold over\/ $\K,$
and\/ $\cM$ is a finite type moduli\/ $\K$-scheme of simple coherent
sheaves in $\coh(Y),$ or simple complexes of coherent sheaves in
$D^b\coh(Y),$ with obstruction theory\/ $\phi:\cE^\bu\ra\bL_\cM$ as
in Thomas\/ {\rm\cite{Thom}} or Huybrechts and Thomas\/
{\rm\cite{HuTh}}. Suppose we are given a square root\/
$\det(\cE^\bu)^{1/2}$ for $\det(\cE^\bu)$. Then we may define a
natural motive $MF_\cM\in\oM^{\hat\mu}_\cM.$
\label{mo5cor4}
\end{cor}

\begin{cor} Let\/ $(S,\om)$ be an algebraic symplectic manifold and\/
$L,M$ finite type algebraic Lagrangian submanifolds in $S,$ and
write $X=L\cap M,$ as a subscheme of\/ $S$. Suppose we are given
square roots\/ $K_L^{1/2},K_M^{1/2}$ for $K_L,K_M$. Then we may
define a natural motive $MF_{L,M}\in \oM^{\hat\mu}_X.$
\label{mo5cor5}
\end{cor}

Corollary \ref{mo5cor4} has applications to Donaldson--Thomas
theory. If $Y$ is a Calabi--Yau 3-fold over $\C$ and $\tau$ a
suitable stability condition on coherent sheaves on $M$, the {\it
Donaldson--Thomas invariants\/} $DT^\al(\tau)$ are integers which
`count' the moduli schemes $\cM_{\rm st}^\al(\tau)$ of $\tau$-stable
coherent sheaves on $Y$ with Chern character $\al\in H^{\rm
even}(Y;\Q)$, provided there are no strictly $\tau$-semistable
sheaves in class $\al$ on $Y$. They were defined by
Thomas~\cite{Thom}.

Behrend \cite{Behr} showed that $DT^\al(\tau)$ may be written as a
weighted Euler characteristic $\chi(\cM_{\rm st}^\al(\tau),\nu)$,
where $\nu:\cM_{\rm st}^\al(\tau)\ra\Z$ is a certain constructible
function called the {\it Behrend function}. Joyce and Song
\cite{JoSo} extended the definition of $DT^\al(\tau)$ to classes
$\al$ including $\tau$-semistable sheaves (with
$DT^\al(\tau)\in\Q$), and proved a wall-crossing formula for
$DT^\al(\tau)$ for change of stability condition $\tau$. Kontsevich
and Soibelman \cite{KoSo1} gave a (partly conjectural) motivic
generalization of Donaldson--Thomas invariants, for Calabi--Yau
3-folds over $\K$, also with a wall-crossing formula.

Kontsevich and Soibelman define a motive over $\cM_{\rm
st}^\al(\tau)$, by associating a formal power series to each (not
necessarily closed) point, and taking its motivic Milnor fibre. The
question of how these formal power series and motivic Milnor fibres
vary in families over the base $\cM_{\rm st}^\al(\tau)$ is not
really addressed in \cite{KoSo1}. Corollary \ref{mo5cor4} answers
this question, showing that Zariski locally in $\cM_{\rm
st}^\al(\tau)$ we can take the formal power series and motivic
Milnor fibres to all come from a regular function $f:U\ra\bA^1$ on a
smooth $\K$-scheme~$U$.

The square root $\det(\cE^\bu)^{1/2}$ required in Corollary
\ref{mo5cor4} corresponds roughly to {\it orientation data\/} in
Kontsevich and Soibelman \cite[\S 5]{KoSo1}, \cite{KoSo2}.

\subsection{Torus localization of motives}
\label{mo53}

We will explain a torus localization formula for the motives
$MF_{X,s}$ of Theorem \ref{mo5thm5}, due to Davesh Maulik
\cite{Maul}. We need the following notation:

\begin{dfn} Suppose $(X,s)$ is a finite type algebraic d-critical
locus over $\K$. Let the multiplicative group $\bG_m=\K\sm\{0\}$ act
on the $\K$-scheme $X$, and write the action as $\rho(\la):X\ra X$
for $\la\in\bG_m$. Then
\begin{itemize}
\setlength{\itemsep}{0pt}
\setlength{\parsep}{0pt}
\item[(i)] We say that $s$ is $\bG_m$-{\it invariant}, and
$(X,s)$ is a $\bG_m$-{\it invariant d-critical locus}, if
$\rho(\la)^\star(s)=s$ for all $\la\in\bG_m$, where the
pullback\/ $\rho(\la)^\star$ is as in
Joyce~\cite[Prop.~2.3]{Joyc2}.
\item[(ii)] We say the $\bG_m$-action $\rho$ on $X$ is {\it
good\/} if $X$ may be covered by Zariski open, affine,
$\bG_m$-invariant $\K$-subschemes $U\subseteq X$.
\item[(iii)] We say the $\bG_m$-action on $X$ is {\it
circle-compact\/} if for all $x\in X$ the limit $\lim_{\la\ra
0}\rho(\la)x$ exists in $X$, where $\la\ra 0$ in $\bG_m=\K\sm
\{0\}\subset\K$. If $X$ is proper, then any $\bG_m$-action is
circle-compact.
\end{itemize}

Suppose $s$ is $\bG_m$-invariant, and the $\bG_m$-action $\rho$ is
good. Then Joyce \cite[Prop.s 2.43 \& 2.44]{Joyc2} shows that
$(X,s)$ admits a cover by $\bG_m$-equivariant critical charts, and
that two such charts can be compared on their overlap by
$\bG_m$-equivariant embeddings into a third $\bG_m$-equivariant
critical chart.

Write $X^{\bG_m}$ for the $\bG_m$-fixed subscheme of $X$, so that
$X^{\bG_m}\subseteq X$ is a finite type closed $\K$-subscheme, with
inclusion $\io:X^{\bG_m}\hookra X$. Write $s^{\bG_m}$ for the
pullback $\io^\star(s)\in H^0(\cSz_{X^{\bG_m}})$ of $s$ to
$X^{\bG_m}$. By \cite[Cor.~2.45]{Joyc2} $(X^{\bG_m},s^{\bG_m})$ is
an algebraic d-critical locus, and one can also show that each
$\bG_m$-equivariant orientation $K_{X,s}^{1/2}$ for $(X,s)$ induces
a natural orientation $K_{X^{\bG_m},s^{\bG_m}}^{1/2}$
for~$(X^{\bG_m},s^{\bG_m})$.

Write $X^{\bG_m}=\coprod_{i\in I}X^{\bG_m}_i$ for the decomposition
of $X^{\bG_m}$ into connected components, where $I$ is a finite
indexing set, and set $s^{\bG_m}_i=s^{\bG_m}
\vert_{\smash{X^{\bG_m}_i}},$ so that $(X^{\bG_m}_i,s^{\bG_m}_i)$ is
a connected algebraic d-critical locus.

Following Maulik \cite[\S 4]{Maul}, define the {\it virtual index\/}
$\mathop{\rm ind}^{\rm vir}(X_i^{\bG_m},X)$ of $X_i^{\bG_m}$ as
follows: if $x\in X_i^{\bG_m}$ is a $\K$-point then $T_xX$ is a
$\K$-vector space with a $\bG_m$-action induced by $\rho$, so we can
split $T_xX=(T_xX)_0\op (T_xX)_+\op (T_xX)_-$, where
$(T_xX)_0,(T_xX)_+,(T_xX)_-$ are the direct sums of eigenspaces of
the $\bG_m$ action on which $\bG_m$ acts with zero weight, or
positive weight, or negative weight, respectively. Then
$(T_xX)_0=T_xX_i^{\bG_m}$. We define $\mathop{\rm ind}^{\rm
vir}(X_i^{\bG_m},X)=\dim (T_xX)_+-\dim (T_xX)_-$. This is
independent of the choice of~$x\in X_i^{\bG_m}$.
\label{mo5def3}
\end{dfn}

\begin{thm}[Maulik \cite{Maul}] Suppose $(X,s)$ is an oriented
d-critical locus, and\/ $\rho$ is a good, circle-compact\/
$\bG_m$-action on $X$ which fixes $s$ and preserves the orientation
$K_{X,s}^{1/2}$. Then as above, the $\bG_m$-invariant subscheme
$X^{\bG_m}$ extends to an oriented d-critical locus
$(X^{\bG_m},s^{\bG_m})=\coprod_{i\in I}(X^{\bG_m}_i,s^{\bG_m}_i)$.

Theorem {\rm\ref{mo5thm5}} gives relative motives
$MF_{X,s}\in\oM^{\hat\mu}_X,$ $MF_{\smash{X^{\bG_m}_i,s^{\bG_m}_i}}
\in\oM^{\smash{\hat\mu}}_{X^{\smash{\bG_m}}_i}$. Writing $\pi:X\ra
*=\Spec\K$ and\/ $\pi:X^{\bG_m}_i\ra *$ for the projections, we have
absolute motives $\pi_*(MF_{X,s}),\pi_*(MF_{X^{\bG_m}_i,
s^{\bG_m}_i})\in\oM^{\hat\mu}_\K$. These are related by
\e
\pi_*(MF_{X,s})=\ts\sum_{i\in I}\bL^{-\mathop{\rm ind}^{\rm
vir}(X_i^{\bG_m},X)/2}\od\pi_*(MF_{\smash{X^{\bG_m}_i,s^{\bG_m}_i}}).
\label{mo5eq6}
\e
In the special case in which $X^{\bG_m}$ consists (as a scheme) only
of finitely many isolated points, equation \eq{mo5eq6} reduces to
\e
\pi_*(MF_{X,s})=\ts\sum_{x\in X^{\bG_m}} \bL^{-\mathop{\rm ind}^{\rm
vir}(\{x\},X)/2}.
\label{mo5eq7}
\e

\label{mo5thm6}
\end{thm}

As in \cite{Maul}, equations \eq{mo5eq6}--\eq{mo5eq7} are powerful
tools for computing the absolute motives $\pi_*(MF_{X,s})\in
\oM^{\hat\mu}_\K$ in examples. Maulik first proves a torus
localization formula for $\bG_m$-equivariant motivic vanishing
cycles, and then deduces Theorem \ref{mo5thm6} using results of
\cite[\S 2.6]{Joyc2}. It seems likely that Theorem \ref{mo5thm6}
also holds without the assumption that the $\bG_m$-action $\rho$ is
good.

\subsection{Proof of Theorem \ref{mo5thm5}}
\label{mo54}

Let $(X,s)$ be an algebraic d-critical locus with orientation
$K_{X,s}^{1/2}$. We must construct $MF_{X,s}\in\oM^{\hat\mu}_X$
satisfying \eq{mo5eq5} for each critical chart $(R,U,f,i)$. Since
such $R\subseteq X$ form a Zariski open cover of $X$, and
\eq{mo5eq5} determines $MF_{X,s}\vert_R$, there exists a unique
$MF_{X,s}$ satisfying \eq{mo5eq5} for all $(R,U,f,i)$ if and only if
the prescribed values $MF_{X,s}\vert_R$ agree on overlaps between
critical charts. That is, we must prove that if $(R,U,f,i)$ and
$(S,V,g,j)$ are critical charts, then
\e
\bigl[i^*\bigl(MF_{U,f}^{\rm mot,\phi}\bigr)\od\Up
(Q_{R,U,f,i})\bigr]\big\vert_{R\cap S}=\bigl[j^*\bigl(MF_{V,g}^{\rm
mot,\phi}\bigr)\od\Up (Q_{S,V,g,j})\bigr]\big\vert_{R\cap S}.
\label{mo5eq8}
\e

Fix $x\in R\cap S\subseteq X$, and let
$(R',U',f',i'),(S',V',g',j'),(T,W,h,k),\Phi,\Psi$ be as in Theorem
\ref{mo5thm2}. Then as in \cite[Th.~6.9(ii)]{BBDJS} there is a
natural isomorphism of principal $\Z_2$-bundles on $R'$
\e
\La_\Phi:Q_{T,W,h,k}\vert_{R'}\,{\buildrel\cong\over\longra}\,
i\vert_{R'}^*(P_\Phi)\ot_{\Z_2} Q_{R,U,f,i}\vert_{R'},
\label{mo5eq9}
\e
for $P_\Phi\ra \Crit(f')$ the principal $\Z_2$-bundle of
orientations of $\bigl(N_{\sst U'W}\vert_{\Crit(f')},q_{\sst
U'W}\bigr)$ as in Definition \ref{mo4def1}, defined as follows:
local isomorphisms
\begin{gather*}
\al:K_{X,s}^{1/2}\vert_{R^{\prime\red}}\longra
i^*(K_U)\vert_{R^{\prime\red}},\quad
\be:K_{X,s}^{1/2}\vert_{R^{\prime\red}}\longra k^*(K_W)
\vert_{R^{\prime\red}},\\
\text{and}\qquad\ga:i^*(K_U)\vert_{R^{\prime\red}}\longra
k^*(K_W)\vert_{R^{\prime\red}}
\end{gather*}
with $\al\ot\al=\io_{R,U,f,i}\vert_{R^{\prime\red}},$
$\be\ot\be=\io_{T,W,h,k}\vert_{R^{\prime\red}},$
$\ga\ot\ga=i\vert_{R^{\prime\red}}^*(J_\Phi)$ correspond to local
sections $s_\al:R'\ra Q_{R,U,f,i}\vert_{R'},$ $s_\be:R'\ra
Q_{T,W,h,k}\vert_{R'},$ $s_\ga:R'\ra i\vert_{R'}^*(P_\Phi)$.
Equation \eq{mo5eq4} shows that $\be=\ga\ci\al$ is a possible
solution for $\be$, and we define $\La_\Phi$ in \eq{mo5eq9} such
that $\La_\Phi(s_\be)=s_\ga\ot_{\Z_2}s_\al$ if and only
if~$\be=\ga\ci\al$.

We now have
\ea
\bigl[k^*&\bigl(MF_{W,h}^{\rm mot,\phi}\bigr)\od\Up
(Q_{T,W,h,k})\bigr]\big\vert_{R'}
\nonumber\\
&=i^{\prime *}\bigl[\Phi\vert_{\Crit(f')}^*\bigl(MF_{W,h}^{\rm
mot,\phi}\bigr)\bigr]\od\Up (Q_{T,W,h,k})\vert_{R'}
\nonumber\\
&=i^{\prime *}\bigl[MF^{\rm mot,
\phi}_{U',f'}\od \Up (P_\Phi)\bigr]\od\Up (Q_{T,W,h,k})\vert_{R'}
\nonumber\\
&=i\vert_{R'}^*\bigl[MF^{\rm mot,\phi}_{U,f}\bigr]\od
\Up\bigl(i\vert_{R'}^*(P_\Phi)\bigr)\od\Up\bigl(Q_{T,W,h,k}\vert_{R'}\bigr)
\nonumber\\
&=i\vert_{R'}^*\bigl[MF^{\rm mot,\phi}_{U,f}\bigr]\od
\Up\bigl(i\vert_{R'}^*(P_\Phi)\ot_{\Z_2}Q_{T,W,h,k}\vert_{R'}\bigr)
\nonumber\\
&=i\vert_{R'}^*\bigl[MF^{\rm mot,\phi}_{U,f}\bigr]\od
\Up\bigl(Q_{R,U,f,i}\vert_{R'}\bigr)
\nonumber\\
&=\bigl[i^*\bigl(MF_{U,f}^{\rm mot,\phi}\bigr)\od\Up
(Q_{R,U,f,i})\bigr]\big\vert_{R'},
\label{mo5eq10}
\ea
using $\Phi\vert_{\Crit(f')}\ci i'=k\vert_{R'}$ in the first step,
Theorem \ref{mo4thm2} for $\Phi:(U',f')\ra (W,h)$ in the second,
$U'\subseteq U$, $f'=f\vert_{U'}$ and functoriality of $\Up$ in the
third, \eq{mo2eq19} in the fourth, and \eq{mo5eq9} in the fifth.

Similarly, from $\Psi:(S',V',g',j')\hookra(T,W,h,k)$ we obtain
\e
\bigl[k^*\bigl(MF_{W,h}^{\rm mot,\phi}\bigr)\od\Up
(Q_{T,W,h,k})\bigr]\big\vert_{S'}= \bigl[j^*\bigl(MF_{V,g}^{\rm
mot,\phi}\bigr)\od\Up (Q_{S,V,g,j})\bigr]\big\vert_{S'}.
\label{mo5eq11}
\e
Combining the restrictions of \eq{mo5eq10}--\eq{mo5eq11} to $R'\cap
S'$ proves the restriction of \eq{mo5eq8} to $R'\cap S'$. Since we
can cover $R\cap S$ by such Zariski open $R'\cap S'\subseteq R\cap
S$, this proves \eq{mo5eq8}, and hence Theorem~\ref{mo5thm5}.

\medskip

\noindent{\small\sc Address for Dominic Joyce:

\noindent The Mathematical Institute, Radcliffe Observatory Quarter,
Woodstock Road, Oxford, OX2 6GG, U.K.

\noindent E-mail: {\tt joyce@maths.ox.ac.uk}.
\smallskip

\noindent Address for Vittoria Bussi:

\noindent ICTP, Strada Costiera 11, Trieste, Italy.

\noindent E-mail: {\tt vbussi@ictp.it}.
\smallskip

\noindent Address for Sven Meinhardt:

\noindent School of Mathematics and Statistics, Hicks Building, Hounsfield Road, Sheffield, S3 7RH, U.K.

\noindent E-mail: {\tt S.Meinhardt@shef.ac.uk}. }

\end{document}